\documentclass[reqno,11 pt]{amsart}
\usepackage{amsmath,amssymb,latexsym}
\usepackage[colorlinks=true, linkcolor={black}]{hyperref}
\usepackage{amsthm}
\usepackage{color}  
\usepackage[all]{xy}
\usepackage{tikz-cd}
\usepackage{multirow}
\usepackage{multicol}
\usepackage{longtable}
\usepackage{caption}
\usetikzlibrary{cd}
\hypersetup{
    colorlinks=true, 
    linktoc=all,     
    linkcolor=blue,  
    urlcolor=black,
}

\newcommand{\F}{{\mathbb{F}}}
\newcommand{\Z}{{\mathbb{Z}}}
\newcommand{\Q}{{\mathbb{Q}}}

\newcommand{\q}{{\mathfrak{q}}}
\newcommand{\p}{{\mathfrak{p}}}

\newcommand{\OO}{{\mathcal{O}}}

\makeatletter
\newcommand*{\rom}[1]{\expandafter\@slowromancap\romannumeral #1@}
\makeatother

\DeclareFontFamily{U}{wncy}{}
    \DeclareFontShape{U}{wncy}{m}{n}{<->wncyr10}{}
    \DeclareSymbolFont{mcy}{U}{wncy}{m}{n}
    \DeclareMathSymbol{\Sh}{\mathord}{mcy}{"58}
		
\theoremstyle{plain}

\newtheorem{theorem}{Theorem}[section]
\newtheorem*{theorem*}{Theorem}
\newtheorem*{conj*}{Sylvester's Conjecture}
\newtheorem*{conj'}{Generalised Sylvester's Conjecture}

\newtheorem{proposition}[theorem]{Proposition}
\newtheorem{rem}[theorem]{Remark}
\newtheorem{lemma}[theorem]{Lemma}
\newtheorem{corollary}[theorem]{Corollary}
\newtheorem{defn}[theorem]{Definition}

\newtheorem{conj}[theorem]{Conjecture}

\begin{document}

\title{Cube sum problem for integers having exactly two distinct prime factors} 

\author{Dipramit Majumdar and Pratiksha Shingavekar}
\address{\begin{scriptsize}Dipramit Majumdar, \href{mailto:dipramit@iitm.ac.in}{dipramit@iitm.ac.in}, IIT Madras, India\end{scriptsize}}
\address{\begin{scriptsize}Pratiksha Shingavekar, \href{mailto:pshingavekar@gmail.com}{pshingavekar@gmail.com}, IIT Madras, India\end{scriptsize}}
\maketitle

\begin{abstract} 
  Given an integer $n>1$, it is a classical Diophantine problem that whether $n$ can be written as a sum of two rational cubes. The study of this problem, considering several special cases of $n$, has a copious history that can be traced back to the works of Sylvester, Satg{\'e}, Selmer etc. and up to the recent works of Alp{\"o}ge-Bhargava-Shnidman. In this article, we consider the cube sum problem for cube-free integers $n$ which has two distinct prime factors none of which is $3$. 
\end{abstract}

\section*{Introduction}\label{intro}
It is a long been asked Diophantine question: which cube-free integers \begin{small}$n>1$\end{small} can be expressed as a sum of cubes of two rational numbers? If \begin{small}$n>2$\end{small} is a cube-free integer and \begin{small}$\text{rk } E_{-432n^2}(\Q)$\end{small} denotes the Mordell-Weil rank of the elliptic curve \begin{small}$E_{-432n^2}:y^2=x^3-432n^2$\end{small} over \begin{small}$\Q$\end{small}, then it is a well-known fact that \begin{small}$n$\end{small} can be written as a sum of two rational cubes if and only if \begin{small}$\text{rk } E_{-432n^2}(\Q) > 0$.\end{small} For a prime \begin{small}$\ell \ge 5$,\end{small} a conjecture classically attributed to Sylvester (see for example \cite{dv}) predicts that
\begin{small} $$\text{rk } E_{-432\ell^2}(\Q)=\begin{cases} 0, & \text{if } \ell \equiv 2,5 (\text{mod }9),\\ 1, & \text{if } \ell \equiv 4,7,8 (\text{mod }9),\\ 0 \text{ or } 2, & \text{if } \ell \equiv 1 (\text{mod }9), \end{cases} \enspace 
\text{rk } E_{-432\ell^4}(\Q)=\begin{cases} 0, & \text{if } \ell \equiv 2,5 (\text{mod }9),\\ 1, & \text{if } \ell \equiv 4,7,8 (\text{mod }9),\\ 0 \text{ or } 2, & \text{if } \ell \equiv 1 (\text{mod }9). \end{cases}$$ \end{small}

One can reformulate the conjecture above as follows:
\begin{conj*}\label{sylconj}
    Let \begin{small}$\ell \ge 5$\end{small} be a prime and either \begin{small}$n = \ell$\end{small} or \begin{small}$n = \ell^2$.\end{small} We have

 \begin{enumerate}
              \item If \begin{small}$n \not\equiv \pm 1 \pmod{9}$,\end{small} then 
            \begin{small}$\text{rk } E_{-432n^2}(\Q) = \begin{cases}
            0, & \text{ if } \ell \equiv 2 \pmod{3},\\
            1, & \text{ if } \ell \equiv 1 \pmod{3}.
            \end{cases}$\end{small}
            
            \item If \begin{small}$n \equiv \pm 1 \pmod{9}$,\end{small} then 
            \begin{small}$\text{rk } E_{-432n^2}(\Q) = \begin{cases}
            1, & \text{ if } \ell \equiv 2 \pmod{3},\\
            \{0, 2 \}, & \text{ if } \ell \equiv 1 \pmod{3}.
            \end{cases}$    \end{small}      
\end{enumerate}
\end{conj*}

Several studies of the cube sum problem considering different particular cases for \begin{small}$n$\end{small} have been done by many mathematicians including Sylvester\cite{syl}, Satg{\'e}\cite{sa}, Selmer\cite{sel} and many others. For example, the case when \begin{small}$n=\ell$\end{small} is a prime such that \begin{small}$\ell \equiv 2 \text{ or } 5 \pmod 9$,\end{small} then the conjecture above was proved by Sylvester himself. If \begin{small}$\ell \equiv 4 \text{ or } 7 \pmod 9$\end{small}, then \cite{dv} proved that \begin{small}$\ell$\end{small} and \begin{small}$\ell^2$\end{small} are cube sums under the hypothesis that $3$ is not a cube modulo \begin{small}$\ell$.\end{small} 
In \cite{ms}, the authors showed that there are infinitely many primes congruent to \begin{small}$\pm 1 \pmod 9$\end{small} which are sum of two rational cubes. Many other authors have contributed towards the Sylvester's conjecture and the cube sum problem, for specific congruence classes of primes modulo $9$ \cite{sy}, \cite{yi}. However, the problem is still open in full generality. In \cite{jms}, we proved Sylvester's Conjecture under the hypothesis that the Tate-Shafarevich group \begin{small}$\Sh(E_{16n^2}/\Q)$\end{small} 
 satisfies \begin{small}$\dim_{\F_3} \Sh(E_{16n^2}/\Q)[3]$\end{small} is even.

For a general cube-free integer \begin{small}$n$,\end{small} recently, \cite{abs} announced a proof of the fact that a positive proportion of integers are cube sums and a positive proportion are not. Apart from this, very little is known about this problem in general. If \begin{small}$n$\end{small} has two distinct prime factors, we know some partial results. Satg{\'e}\cite{sa} proved that \begin{small}$n=2\ell$\end{small} (resp. \begin{small}$n=2\ell^2$\end{small}) is a cube sum when \begin{small}$\ell \equiv 2 \pmod 9$\end{small} (resp. \begin{small}$\ell \equiv 5 \pmod 9$\end{small}); \cite{cst} and \cite{kl} also study this situation. In \cite{jms}, we have also examined the integers of type \begin{small} $2\ell$\end{small} and \begin{small}$2\ell^2$\end{small} for their solubility as a cube sum, where \begin{small}$\ell \neq 3$\end{small} is a prime. In this article, we generalize the situation to consider the cube sum problem for cube-free integers \begin{small}$n$\end{small} which are co-prime to $3$ and have exactly two distinct prime factors, that is \begin{small}$n=\ell_1\ell_2$\end{small} or \begin{small}$n=\ell_1^2\ell_2^2$\end{small} or \begin{small}$n=\ell_1^2\ell_2$,\end{small} where \begin{small}$\ell_1$\end{small} and \begin{small}$\ell_2$\end{small} are two distinct primes, none of which is equal to $3$.

We start by stating a conjecture which is slightly more general in nature. 
By the work of Birch-Stephens \cite[\S4]{bs}, we know that for a cube-free \begin{small}$n$,\end{small} the global root number of \begin{small}$E_{-432n^2}$\end{small} is given by \begin{small}$-\omega_3 \prod_{p \neq 3} \omega_p$.\end{small} In particular, when \begin{small}$n$\end{small} is a cube-free integer co-prime to $3$, this formula for the global root number simplifies to  \begin{small}$(-1)^{k_2}$\end{small} (resp. to  \begin{small}$(-1)^{1+k_2}$\end{small}) if \begin{small}$n \equiv \pm 1 \pmod 9$\end{small} (resp. if \begin{small}$n \not\equiv \pm 1 \pmod 9$\end{small}); here \begin{small}$k_2$\end{small} is the number of distinct primes \begin{small}$p$\end{small} which divide \begin{small}$n$\end{small} and are congruent to $2$ modulo $3$. Looking at the sign of the functional equation of \begin{small}$E_{-432n^2}$,\end{small} we predict:
\begin{conj'}\label{conjforrank1}
    Let \begin{small}$n$\end{small} be a cube-free integer co-prime to $3$. Assume that the set of prime divisors of  \begin{small}$n$\end{small} is given by \begin{small}$\{ \ell_1, \ell_2, \dots, \ell_{k_1+k_2}\}, $\end{small} where \begin{small}$\ell_1 \equiv \cdots \equiv \ell_{k_1} \equiv 1 \pmod 3$\end{small} and \begin{small}$\ell_{k_1+1} \equiv \cdots \equiv \ell_{k_1+k_2} \equiv 2 \pmod 3$.\end{small} Let \begin{small}$r_n := \text{rk } E_{-432n^2}(\Q)$.\end{small} We have

 \begin{enumerate}
             \item If \begin{small}$n \not\equiv \pm 1 \pmod{9}$,\end{small} then \begin{small}$r_n \le 2k_1+k_2-1$\end{small} and  \begin{small}$r_n \equiv k_2-1 \pmod 2$.\end{small} 
            \item If \begin{small}$n \equiv \pm 1 \pmod{9}$,\end{small} then \begin{small}$r_n \le 2k_1+k_2$\end{small} and \begin{small}$r_n \equiv k_2 \pmod 2$\end{small}. 
\end{enumerate}
\end{conj'}

\noindent {\bf A heuristic for the Generalised Sylvester's Conjecture: \\}  
Consider the elliptic curves  \begin{small}$E_{16n^2}:y^2=x^3+16n^2$\end{small} and  \begin{small}$E_{-432n^2}:y^2=x^3-432n^2$\end{small} defined over \begin{small}$\Q$.\end{small} It is a well-known fact that there exists a rational $3$-isogeny \begin{small}$\phi_n: E_{16n^2} \to E_{-432n^2}$\end{small} given by \begin{small}$(x,y) \mapsto \big(\frac{x^3+64n^2}{x^2},\frac{y(x^3-128n^2)}{x^3}\big)$.\end{small} This implies that \begin{small}$\text{rk } E_{-432n^2}(\Q)= \text{rk } E_{16n^2}(\Q)$.\end{small} Hence, for determining \begin{small}$r_n$\end{small}, it is enough to obtain \begin{small}$\text{rk } E_{16n^2}(\Q)$.\end{small}

 Let \begin{small}$K:=\Q(\zeta)$\end{small} and \begin{small}$\p:=1-\zeta \in K$,\end{small} where \begin{small}$\zeta=\frac{-1+\sqrt{-3}}{2}$\end{small} denotes the primitive third root of unity. Observe that  \begin{small}$\p^6=-27$\end{small} and so, it follows that \begin{small}$E_{16n^2}$\end{small} and \begin{small}$E_{-432n^2}$\end{small} are isomorphic over \begin{small}$K$.\end{small} Composing \begin{small}$\phi_n$\end{small} with this isomorphism, we obtain a $3$-isogeny  \begin{small}$\phi:E_{16n^2} \to E_{16n^2}$\end{small} defined over \begin{small}$K$\end{small} given by \begin{small}\begin{equation}\label{eq:defofphi}
\phi(x,y)=\Big(\frac{x^3+64n^2}{\p^2x^2},\frac{y(x^3-128n^2)}{\p^3x^3}\Big).
\end{equation}\end{small}  

Further, let \begin{small}$\Sigma_K$\end{small} be the set of all finite primes of \begin{small}$K$\end{small} and \begin{small}$S_{n}:=\{\q \in \Sigma_K \text{ } \big| \quad \q \mid n\}.$\end{small} Note that  \begin{small}$|S_n|=2k_1+k_2$.\end{small} It follows from Corollary \ref{corforrk} that \begin{small}$r_n \le 2k_1+k_2-1$\end{small} (resp. \begin{small}$r_n \le 2k_1+k_2$\end{small}) if \begin{small}$n \not\equiv \pm 1 \pmod 9$\end{small} (resp. if \begin{small}$n \equiv \pm 1 \pmod 9$\end{small}). From these upper bounds on the Mordell-Weil rank of \begin{small}$E_{-432n^2}$\end{small} over \begin{small}$\Q$\end{small} along with the global root number, the parity conjecture implies Generalised Sylvester's Conjecture.\\

 Sylvester's Conjecture is a special case of the Generalised Sylvester's Conjecture for \begin{small}$k_1+k_2=1$.\end{small} In this paper, we study the case \begin{small}$k_1+k_2=2$,\end{small} that is, when \begin{small}$n$\end{small} has exactly two distinct prime divisors both different from $3$. In this special case, the Generalised Sylvester's Conjecture can be restated as follows:
\begin{conj}\label{conjforrank2}
    Let \begin{small}$\ell_1$\end{small} and \begin{small}$\ell_2$\end{small} be two distinct primes both different from $3$. Assume that \begin{small}$n$\end{small} is of the form either \begin{small}$n = \ell_1\ell_2$\end{small} or \begin{small}$n = \ell_1^2\ell_2$\end{small} or \begin{small}$n = \ell_1^2\ell_2^2$. \end{small} We have

 \begin{enumerate}
             \item If \begin{small}$n \not\equiv \pm 1 \pmod{9}$,\end{small} then 
            \begin{small}$\text{rk } E_{-432n^2}(\Q) = \begin{cases}
            \{0, 2\}, & \text{ if } \ell_1\ell_2 \equiv 2 \pmod{3},\\
            1, & \text{ if } \ell_1 \equiv \ell_2 \equiv 2 \pmod{3},\\
            \{1,3\}, & \text{ if } \ell_1 \equiv \ell_2 \equiv 1 \pmod{3}.\end{cases}$\end{small}

            \item If \begin{small}$n \equiv \pm 1 \pmod{9}$,\end{small} then 
            \begin{small}$\text{rk } E_{-432n^2}(\Q) = \begin{cases}
            \{1,3\}, & \text{ if } \ell_1\ell_2 \equiv 2 \pmod{3},\\
            \{0, 2 \}, & \text{ if } \ell_1 \equiv \ell_2 \equiv 2 \pmod{3},\\
            \{0, 2, 4 \}, & \text{ if } \ell_1 \equiv \ell_2 \equiv 1 \pmod{3}.\end{cases}$\end{small}     
\end{enumerate}
\end{conj}

To prove Conjecture \ref{conjforrank2}, we study the \begin{small}$\F_3$\end{small}-dimension of the \begin{small}$\phi$\end{small}-Selmer group of \begin{small}$E_{16n^2}$\end{small} over \begin{small}$K$\end{small} and prove that it has the same parity as \begin{small}$k_2$\end{small} (resp. \begin{small}$k_2+1$\end{small}) provided \begin{small}$n \not\equiv \pm 1 \pmod 9$\end{small} (resp. \begin{small}$n \equiv \pm 1 \pmod 9$\end{small}). Now a simple application of Lemma \ref{rankbound} proves Conjecture \ref{conjforrank2} (assuming \begin{small}$\dim_{\F_3} \Sh(E_{16n^2}/\Q)[3]$\end{small} is even). We remark that the Tate-Shafarevich conjecture predicts that \begin{small}$\Sh(E_{16n^2}/\Q)$\end{small} is finite, which in turn implies that \begin{small}$\dim_{\F_3} \Sh(E_{16n^2}/\Q)[3]$\end{small} is even (see \cite[Remark~5.13]{jms} for a more detailed discussion). 

To compute \begin{small}$\dim_{\F_3} {\rm Sel}^\phi(E_{16n^2}/K)$\end{small} explicitly in all possible scenarios, we broadly divide them in three cases, namely \begin{small}
    $k_1=0$, $k_1=1$ \text{ and } $k_1=2$.
\end{small} The computations for each of these cases are quite involved and challenging. From these computations we observe that \begin{small}$\dim_{\F_3} {\rm Sel}^\phi(E_{16n^2}/K)$\end{small} depends on equivalence classes of \begin{small}$\ell_1, \ell_2, n$\end{small} modulo \begin{small}$9$\end{small}, as well as on certain  cubic residues (if \begin{small}$k_1 >0$\end{small}). See \S\ref{sec2} for precise statements. Finally, Lemma \ref{rankbound} provides us with bounds on \begin{small}$\text{rk } E_{16n^2}(\Q)$\end{small} and we show that our bounds on \begin{small}$\text{rk } E_{16n^2}(\Q)$\end{small} are sharp. See \S\ref{rankcomp} for precise statements.

One of the consequences of our explicit computation of the Selmer groups is that it leads us to obtain the following unconditional result: 
\begin{theorem*}\label{rank0}
    Let \begin{small}$n$\end{small} be a cube-free integer co-prime to $3$ having  exactly two distinct prime factors \begin{small}$\ell_1$\end{small} and \begin{small}$\ell_2$.\end{small} We have \begin{small}$\text{rk } E_{-432n^2}(\Q)=0$\end{small} i.e. \begin{small}$n$\end{small} is not a rational cube sum in the following cases: 
    \begin{enumerate}
        \item \begin{small}$n=\ell_1\ell_2$\end{small} or \begin{small}$n=\ell_1^2\ell_2^2$\end{small} with \begin{small}$\ell_1 \equiv 2 \pmod 9$\end{small} and \begin{small}$\ell_2 \equiv 5 \pmod 9$,\end{small}

        \item \begin{small}$n=\ell_1^2\ell_2$\end{small} with \begin{small}$\ell_1 \equiv \ell_2 \equiv 2,5 \pmod 9$\end{small},
        
        \item \begin{small}$n=\ell_1\ell_2$\end{small} or \begin{small}$n=\ell_1^2\ell_2$\end{small} with  \begin{small}$n \equiv 2,5 \pmod 9$\end{small} where \begin{small}$\ell_1 \equiv 1 \pmod 3, \ell_2 \equiv 2 \pmod 3$\end{small} and \begin{small}$\ell_2$\end{small} is not a cube in \begin{small}$\F_{\ell_1}$\end{small},

        \item \begin{small}$n=\ell_1^2\ell_2^2$\end{small} or \begin{small}$n=\ell_1\ell_2^2$\end{small} with  \begin{small}$n \equiv 4,7 \pmod 9$\end{small} where \begin{small}$\ell_1 \equiv 1 \pmod 3, \ell_2 \equiv 2 \pmod 3$\end{small} and \begin{small}$\ell_2$\end{small} is not a cube in \begin{small}$\F_{\ell_1}$\end{small}.    
    \end{enumerate} 
\end{theorem*}

Under the assumption \begin{small}$\dim_{\F_3} \Sh(E_{16n^2}/\Q)[3]$\end{small} is even, we also provide a criterion for natural numbers \begin{small}$n \in \{\ell_1\ell_2, \thinspace \ell_1^2\ell_2^2, \thinspace \ell_1^2\ell_2, \thinspace \ell_1\ell_2^2\}$\end{small} for which \begin{small}$\text{rk } E_{16n^2}(\Q)=1$\end{small} (cf. Theorem \ref{thmtosylvesterrk1}) in terms of congruences modulo $9$ and cubic residues.

The article is structured as follows: In \S\ref{prelims}, we discuss the basic set-up, recall some definitions and results from \cite{jms} which are used subsequently. Further, we explore the relationship between dimension of \begin{small}$\phi$\end{small}-Selmer group over \begin{small}$K$\end{small} and rank of elliptic curve over \begin{small}$\Q$.\end{small} This is done in Lemma \ref{rankbound}, which is one of the key components of the paper. In \S\ref{sec2}, we restrict ourselves to cube-free natural numbers co-prime to $3$ with exactly two distinct prime divisors and give an explicit description of the \begin{small}$\phi$\end{small}-Selmer group of \begin{small}$E_{16n^2}/K$\end{small} through computing its generators. Finally in \S\ref{rankcomp}, using Lemma \ref{rankbound}, we relate \begin{small}$\dim_{\F_3} {\rm Sel}^\phi(E_{16n^2}/K)$\end{small} with \begin{small}$\text{rk } E_{16n^2}(\Q)$,\end{small} thereby comment about the solubility of \begin{small}$n$\end{small} as a sum of two rational cubes. \\

\begin{small} {\bf Acknowledgements:} It is a pleasure to thank Prof. Somnath Jha for his encouragement, advice, remarks and several e-mail communications. \\
Both authors were partially supported by  MHRD SPARC grant number 445. 
P. Shingavekar was partially supported by the MHRD grant SB20210807PHMHRD008128.\end{small}

\section{Basic set up }\label{prelims}
First, we fix the notation that we will be using throughout this paper. Let \begin{small}$K=\Q(\zeta)$\end{small} and \begin{small}$\OO_K=\Z[\zeta]$\end{small} denote its ring of integers. Further, let \begin{small}$\Sigma_K$\end{small} be the set of all finite primes of \begin{small}$K$\end{small} and \begin{small}$K_\q$\end{small} denote the completion at \begin{small}$\q \in \Sigma_K$\end{small} of \begin{small}$K$\end{small}.  If we set \begin{small}$\p=1-\zeta$\end{small}, then \begin{small}$(\p)$\end{small} is the unique prime of \begin{small}$K$\end{small} dividing $3$. We denote the ideal \begin{small}$(\p)$\end{small} and its uniformizer \begin{small}$\p$\end{small} both by \begin{small}$\p$\end{small} itself.

There is a norm map \begin{small}$\overline{N}: \frac{K^*}{K^{*3}} \times \frac{K^*}{K^{*3}} \to \frac{K^*}{K^{*3}}$\end{small} given by  \begin{small}$(\overline{x}_1,\overline{x}_2) \mapsto \overline{x}_1\overline{x}_2$.\end{small} We denote the kernel of this norm map by \begin{small}$\big(K^*/K^{*3} \times K^*/K^{*3}\big)_{N=1}$\end{small}. The  map \begin{small}$\overline{N}$\end{small} induces a norm map \begin{small}$\overline{N}:\frac{\OO_K^*}{\OO_K^{*3}} \times \frac{\OO_K^*}{\OO_K^{*3}} \to \frac{\OO_K^*}{\OO_K^{*3}}$.\end{small}
We recall the following from \cite[Definitions~2.1 and ~3.13]{jms}:
\begin{defn}\label{defnofAq}
 We define the following subgroup of \begin{small}${\OO_K^*}/{\OO_K^{*3}} \times {\OO_K^*}/{\OO_K^{*3}}$:\end{small} 
 \begin{small}$$\Big(\frac{A^*}{A^{*3}}\Big)_{N=1}:=\Big\{ (\overline{x}_1,\overline{x}_2) \in \frac{\OO_K^*}{\OO_K^{*3}} \times \frac{\OO_K^*}{\OO_K^{*3}} \text{ } \Big| \quad x_1x_2 \in \OO_K^{*3}, \text{ where } x_i \text{ is any lift of } \overline{x}_i \text{ for } i=1,2 \Big\}.$$\end{small}
The group 
\begin{small}$\big(A_\q^*/A_\q^{*3}\big)_{N=1}$\end{small} is defined similarly by replacing \begin{small}$\OO_K$\end{small} with  \begin{small}$\OO_{K_\q}$\end{small}, the ring of integers of \begin{small}${K_\q}$\end{small}, in the above definition.
\end{defn}

\begin{defn}\label{defofSa}
    Given a cube-free integer \begin{small}$n$,\end{small} we define  
    \begin{small}$S_{n}:=\{\q \in \Sigma_K \text{ } \big| \quad \q \mid n\}.$\end{small}
\end{defn}

\noindent By \begin{small}$\OO_{K,S_n}$\end{small} we denote the ring of \begin{small}$S_n$\end{small}-integers of \begin{small}${K}$\end{small} and  define, in a similar fashion as in the Definition \ref{defnofAq}, the group \begin{small}$\Big(\frac{A_{S_{n}}^*}{A_{S_{n}}^{*3}}\Big)_{N=1}:=\Big(\frac{\OO_{K,S_{n}}^*}{\OO_{K,S_{n}}^{*3}} \times \frac{\OO_{K,S_{n}}^*}{\OO_{K,S_{n}}^{*3}}\Big)_{N=1}$\end{small}.\\

Recall that \begin{small}$\p \in \Sigma_K$\end{small} is the unique of \begin{small}$K$\end{small} above $3$. For the local field \begin{small}$K_\p$\end{small} and its uniformizer \begin{small}$\p$\end{small}, we set \begin{small}$U^m_{K_\p}:=1+\p^m\OO_{K_\p}$.\end{small} Recall the following result from \cite[Lemma~2.7]{jms}:
\begin{lemma}\label{V3U1U4}
We have \begin{small}$\Big(\frac{A_\p^*}{A_\p^{*3}}\Big)_{N=1} \cong \frac{U^1_{K_\p}}{U^4_{K_\p}}$\end{small} and \begin{small}$\frac{V_3}{A_\p^{*3}} \cong \frac{U^3_{K_\p}}{U^4_{K_\p}}$,\end{small} where \begin{small}$V_3 \subset \OO_{K_\p}^* \times \OO_{K_\p}^*$\end{small} is defined
\begin{scriptsize}
         $$V_3:=\{(u_1,u_2) \in \OO_{K_\p}^* \times \OO_{K_\p}^* \mid \overline{u_i} \in \OO_{K_\p}^*/U^3_{K_\p} \text{ satisfies } \overline{u_i}=\alpha_i^3 \text{ for some } \alpha \in \OO_{K_\p}^*/U^3_{K_\p}, i=1,2 \text{ and } u_1u_2 \in \OO_{K_\p}^{*3} \}.$$
    \end{scriptsize}
\end{lemma}
\begin{proof}
    Note that \begin{small}$\big(U^1_{K_\p}\big)^3=U^4_{K_\p}$.\end{small} The map \begin{small}$f: \frac{U^1_{K_\p}}{U^4_{K_\p}} \to \Big(\frac{A_\p^{*}}{A_\p^{*3}}\Big)_{N=1}$\end{small} given by \begin{small}$f(\overline{s}) =(\overline{s}^2, \overline{s})$\end{small} is then an isomorphism. The same map identifies \begin{small}$\frac{U^3_{K_\p}}{U^4_{K_\p}}$\end{small} with \begin{small}$\frac{V_3}{A_\p^{*3}}$.\end{small}
\end{proof}

Further, recall from \cite[Prop.~3.4]{jms} that for the elliptic curve \begin{small}$E_{16n^2}:y^2=x^3+16n^2$\end{small} and a prime \begin{small}$\q \in \Sigma_K$\end{small}, the local Kummer map \begin{small}$\delta_{\phi,K_\q}:E_{16n^2}(K_\q) \to \big({K_\q^*}/{K_\q^{*3}} \times {K_\q^*}/{K_\q^{*3}}\big)_{N=1}$\end{small}, corresponding to the isogeny \begin{small}$\phi$\end{small} in \eqref{eq:defofphi}, is defined as follows: For \begin{small}$P=(x(P),y(P)) \in E_{16n^2}(K_\q)$,\end{small}
\begin{small}\begin{equation}\label{eq:kummap}
    \delta_{\phi,K_\q}(P):=\begin{cases} (\overline{1},\overline{1}), & \text{ if } P=O,\\(\overline{n}^2,\overline{n}), & \text{ if } P=(0,4n),\\ (\overline{n},\overline{n}^2), & \text{ if } P=(0,-4n),\\ (\overline{y(P)-4n}, \overline{y(P)+4n}), & \text{ if } P \notin E_{16n^2}(K_\q)[\phi].    \end{cases}
\end{equation}\end{small}

\begin{lemma}\label{imgofkummap}
    Let \begin{small}$\q \in \Sigma_K \setminus S_n$\end{small}. Then one has \cite[Props.~3.7, ~3.11 and ~3.12]{jms}
    \begin{enumerate}
        \item If \begin{small} $ \q \nmid 3$,\end{small} then \begin{small}$\text{Im } \delta_{\phi,K_\q}= \big(A_\q^*/A_\q^{*3}\big)_{N=1} $ \end{small} and \begin{small} $ |\text{Im } \delta_{\phi,K_\q}|=3$.\end{small} 
        \item If \begin{small} $ \q=\p \mid 3$,\end{small} then \begin{small}$V_3/A_\p^{*3} \subset \text{Im } \delta_{\phi,K_\p} \subset \big(A_\p^*/A_\p^{*3}\big)_{N=1} $ \end{small} and \begin{small} $ |\text{Im } \delta_{\phi,K_\p}|=9$.\end{small} 
    \end{enumerate}
    On the other hand, for \begin{small}$\q \in S_n$,\end{small} \begin{small}$\text{Im } \delta_{\phi,K_\q} = \langle(\overline{n}^2, \overline{n}) \rangle$\end{small}, hence \begin{small}$\text{Im } \delta_{\phi,K_\q} \cap \big(A_\q^*/A_\q^{*3}\big)_{N=1}=\{1\}$\end{small}. 
    \qed
\end{lemma}

Using the explicit description \eqref{eq:kummap} of the Kummer map, we now compute  \begin{small}$\text{Im } \delta_{\phi,K_\p}$\end{small}:
\begin{lemma}\label{imgofdeltaatp}
    Let \begin{small}$n$\end{small} be a cube-free integer co-prime to $3$. We have
    \begin{enumerate}
    \item For \begin{small}$n \not\equiv \pm 1 \pmod 9,$\end{small}         \begin{small}$\text{Im } \delta_{\phi,K_\p} =           \big\langle(\overline{n}^2,\overline{n}),   (\overline{(1+\p^3)^2},\overline{(1+\p^3)}) \big\rangle.$\end{small}  Here, \begin{small}$\text{Im } \delta_{\phi,K_\p} \cong \frac{U^2_{K_\p}}{U^4_{K_\p}}$.\end{small} As a consequence, we see that \begin{small}$(\overline{\zeta^2},\overline{\zeta}) \notin \text{Im } \delta_{\phi,K_\p}$.\end{small}
    
    \item For \begin{small}$n \equiv \pm 1 \pmod 9,$\end{small}
         \begin{small}$\text{Im }\delta_{\phi,K_\p}=\big\langle(\overline{\zeta^2},\overline{\zeta}),
         (\overline{(1+\p^3)^2},\overline{(1+\p^3)}) \big\rangle.$\end{small} In this situation,\\ \begin{small}$\text{Im } \delta_{\phi,K_\p} \cap f\big(\frac{U^2_{K_\p}}{U^4_{K_\p}}\big) \cong \frac{U^3_{K_\p}}{U^4_{K_\p}}$,\end{small} where \begin{small}$f$\end{small} is the isomorphism from Lemma \ref{V3U1U4}.
    
    \end{enumerate}
\end{lemma}
\begin{proof}
(1)  We have \begin{small}$n \equiv 2,4,5,7 \pmod 9$\end{small} and so \begin{small}$n^2 \in 1+3\Z_3 \subset U^2_{K_\p}$,\end{small} but \begin{small}$n^2 \notin  U^3_{K_\p}$.\end{small} Note that \begin{small}$(0,4n) \in E_{16n^2}(\Q)$\end{small} and \begin{small}$\delta_{\phi,K_\p}((0,4n))=\big(\overline{n}^2,\overline{n}\big).$\end{small} Since \begin{small}$\frac{V_3}{A_\p^{*3}}=\big\langle\big(\overline{(1+\p^3)^2},\overline{(1+\p^3)}\big)\big\rangle \subset \text{Im } \delta_{\phi,K_\p}$\end{small} and \begin{small}$\big|\text{Im } \delta_{\phi,K_\p}\big|=9$\end{small} by Lemma \ref{imgofkummap}, we get \begin{small}$\text{Im } \delta_{\phi,K_\p}= \big\langle \big(\overline{n}^2,\overline{n}\big), \big(\overline{(1+\p^3)^2},\overline{(1+\p^3)}\big) \big\rangle$.\end{small}
It is then easy to see that \begin{small}$\text{Im } \delta_{\phi,K_\p} \cong \frac{U^2_{K_\p}}{U^4_{K_\p}}$.\end{small}

\noindent (2) Now if \begin{small}$n \equiv \pm 1 \pmod 9$\end{small}, then \begin{small}$16n^2-1 \equiv -3 \pmod 9$.\end{small} Hence, there exists \begin{small}$\alpha \in 1 +3\Z_3 \subset U^2_{K_\p}$\end{small} such that \begin{small}$16n^2-1 = -3 \alpha^2$.\end{small} Thus, the point \begin{small}$P=(-1,\p\zeta\alpha)=(-1,\sqrt{16n^2-1}) \in E_{16n^2}(K_\p)$.\end{small} Observe that \begin{small}$\delta_{\phi,K_\p}(P)=\big(\overline{\sqrt{16n^2-1} - 4n},\overline{\sqrt{16n^2-1} + 4n}\big)= \big(\overline{\p\zeta\alpha - 4n},\overline{ \p\zeta\alpha+ 4n}\big). $ \end{small} Now \begin{small} $ \p\zeta\alpha+ 4n  \equiv  \p\zeta (1+ \p^2 \beta) \pm 4 \pmod{\p^4} $\end{small} as \begin{small}$\alpha = 1 + \p^2 \beta$\end{small} for some \begin{small}$\beta \in \Z_3$.\end{small} Consequently, \begin{small}$$\p\zeta\alpha+ 4n \equiv \begin{cases} \p \zeta  \pm 4 \equiv \sqrt{-3}  \pm 4 \equiv \pm \big( (1+\p^3)\zeta^2 \big)^{\pm 1}  \pmod{\p^4}, & \text{if } \beta \equiv 0 \pmod{\p}, \\ \p \zeta (1+ \p^2)  \pm 4 \equiv -2 \sqrt{-3}  \pm 4 \equiv \pm \big((1+\p^3)^2\zeta^2 \big)^{\pm 1} \pmod{\p^4}, & \text{if } \beta \equiv 1 \pmod{\p}, \\  \p \zeta(1-\p^2) \pm 4 \equiv \pm \frac{1}{2} - \frac{1}{2} \sqrt{-3}  \equiv \pm \big(\zeta^2)^{\pm 1}\pmod{\p^4}, & \text{if } \beta \equiv -1 \pmod{\p}.\end{cases}$$\end{small}

\noindent As \begin{small}$\frac{V_3}{A_\p^{*3}}=\big\langle\big(\overline{(1+\p^3)^2},\overline{(1+\p^3)}\big)\big\rangle \subset \text{Im } \delta_{\phi,K_\p}$\end{small} and \begin{small}$\big|\text{Im } \delta_{\phi,K_\p}\big|=9$\end{small} by Lemma \ref{imgofkummap}, we see that \begin{small}$\text{Im } \delta_{\phi,K_\p}= \big\langle \big(\overline{\zeta^2},\overline{\zeta}\big),\big(\overline{(1+\p^3)^2},\overline{(1+\p^3)}\big) \big\rangle$.\end{small} 
Note that \begin{small}$\zeta \notin U^2_{K_\p}$;\end{small} so \begin{small}$\text{Im } \delta_{\phi,K_\p} \cap f\big(\frac{U^2_{K_\p}}{U^4_{K_\p}}\big) \cong \frac{U^3_{K_\p}}{U^4_{K_\p}}$.\end{small}
\end{proof}

For the isogeny \begin{small}
    $\phi: E_{16n^2} \to E_{16n^2}$
\end{small}, we define the \begin{small}$\phi$-\end{small}Selmer group of \begin{small}$E_{16n^2}/K$\end{small} as a subgroup of \begin{small}
    $H^1(G_K, E_{16n^2}[\phi]) \cong \big(K^*/K^{*3} \times K^*/K^{*3}\big)_{N=1}$
\end{small}, which we denote by \begin{small}${\rm Sel}^\phi (E_{16n^2}/K)$,\end{small} as follows (cf. \cite[Eq.~(6)]{jms} for details): 
\begin{small}\begin{equation}\label{eq:defofsel}
    {\rm Sel}^\phi (E_{16n^2}/K):=\{(\overline{x}_1,\overline{x}_2) \in \big(K^*/K^{*3} \times K^*/K^{*3}\big)_{N=1} \mid (\overline{x}_1,\overline{x}_2) \in \text{Im } \delta_{\phi,K_\q} \text{ for all } \q \in \Sigma_K\}.
\end{equation}\end{small}

Using the following result, one obtains a trivial bound on \begin{small}$\dim_{\F_3} {\rm Sel}^\phi (E_{16n^2}/K)$\end{small} \cite[Theorem~3.15]{jms}:
\begin{lemma}\label{selmerbound}
    Let \begin{small}$E_{16n^2}:y^2=x^3+16n^2$\end{small} be an elliptic curve and \begin{small}$\phi:E_{16n^2} \to E_{16n^2}$\end{small} be the $3$-isogeny over \begin{small}$K$\end{small} in \eqref{eq:defofphi}. Then \begin{small}${\rm Sel}^\phi (E_{16n^2}/K) \subset \big({A_{S_{n}}^*}/{A_{S_{n}}^{*3}}\big)_{N=1}$\end{small}. As a consequence we obtain that \begin{small} $ \dim_{\F_3} {\rm Sel}^\phi (E_{16n^2}/K) \le \dim_{\F_3} \big({A_{S_{n}}^*}/{A_{S_{n}}^{*3}}\big)_{N=1} = \dim_{\F_3} \frac{\OO_{K,S_{n}}^*}{\OO_{K,S_{n}}^{*3}} = |S_{n}|+1$.\end{small} 
\end{lemma}

Following results relate \begin{small} $ \dim_{\F_3} {\rm Sel}^\phi (E_{16n^2}/K)$\end{small} with \begin{small} $\text{rk } E_{16n^2}(\Q)$\end{small}:
\begin{corollary}\label{corforrk}
Let \begin{small}$n$\end{small} be a cube-free natural number co-prime to $3$ and  \begin{small}$\{ \ell_1, \ldots, \ell_{k_1+k_2} \}$\end{small} be the set of all  distinct primes dividing \begin{small}$n$.\end{small} Further, assume that \begin{small}$\ell_i \equiv 1 \pmod 3$\end{small} for \begin{small}$ i \le k_1$\end{small} and \begin{small}$\ell_j \equiv 2 \pmod 3$\end{small} for \begin{small}$k_1+1 \le j \le k_1+k_2$. \end{small} Then 
\begin{small}$$ r_n := \text{rk } E_{-432n^2}(\Q) \le \begin{cases}
    2k_1+ k_2-1, & \text{ if } n \not\equiv \pm 1 \pmod 9,\\
    2k_1+ k_2, & \text{ if } n \equiv \pm 1 \pmod 9.
\end{cases}$$\end{small}
\end{corollary}

\begin{proof}
Notice that the the primes \begin{small}$\ell_j$\end{small} for \begin{small}$k_1+1 \le j \le k_1+k_2$\end{small} remain inert in \begin{small}$K$,\end{small} whereas the primes \begin{small}$\ell_i$\end{small} for \begin{small}$1 \le i \le k_1$\end{small} split in \begin{small}$K$,\end{small} say \begin{small}$\ell_i\OO_K = \pi_{\ell_i} \pi'_{\ell_i}$.\end{small} (One can make an explicit choice of \begin{small}$\pi_{\ell_i}$\end{small} (resp. \begin{small}$\pi'_{\ell_i}$\end{small}) as follows: Write \begin{small}$\ell_i = \frac{1}{4}(L_i^2+27M_i^2)$\end{small} and define \begin{small}$\pi_{\ell_i} = \frac{1}{2}(L_i+ \sqrt{-27} M_i)$\end{small} and \begin{small}$\pi'_{\ell_i}= \frac{1}{2}(L_i- \sqrt{-27} M_i)$.\end{small}) Therefore,
\begin{small}$$ S_n = \{ \pi_{\ell_1}, \pi'_{\ell_1}, \ldots, \pi_{\ell_{k_1}}, \pi'_{\ell_{k_1}}, \ell_{k_1+1}, \ldots, \ell_{k_1+k_2}  \}$$\end{small}
and hence \begin{small}$|S_n| = 2k_1+k_2$.\end{small} So, it follows from Lemma \ref{selmerbound} that \begin{small}$\dim_{\F_3} {\rm Sel}^\phi (E_{16n^2}/K) \le |S_n| +1 = 2k_1+k_2+1$.\end{small} 
Now recall from \cite[Corollary~3.26]{jms} that
\begin{small}\begin{equation}\label{eq:expforrk} \text{rk } E_{16n^2}(\Q) = \dim_{\F_3} {\rm Sel}^\phi (E_{16n^2}/K) - \dim_{\F_3} E_{16n^2}(K)[\phi] - \dim_{\F_3} \Sh(E_{16n^2}/K)[\phi]\end{equation}\end{small}
Thus, \begin{small}$\text{rk } E_{16n^2}(\Q) \le \dim_{\F_3} {\rm Sel}^\phi (E_{16n^2}/K) - \dim_{\F_3} E_{16n^2}(K)[\phi]$. \end{small} Note that  \begin{small}$|E_{16n^2}(K)[\phi]|=3$\end{small}, so it follows that
 \begin{small}$r_n=\text{rk } E_{16n^2}(\Q) \le \dim_{\F_3} {\rm Sel}^\phi (E_{16n^2}/K)-1 \le |S_n|= 2k_1+k_2$.\end{small}\\

In particular, if \begin{small}$n \not\equiv \pm 1 \pmod 9$,\end{small} then we claim that \begin{small}$\dim_{\F_3} {\rm Sel}^\phi (E_{16n^2}/K) \le |S_n| = 2k_1+k_2$\end{small} and hence \begin{small}$r_n \le 2k_1+k_2-1$\end{small} in this situation. Observe that the group 
\begin{small}$ \big({A_{S_{n}}^*}/{A_{S_{n}}^{*3}}\big)_{N=1} := \{ (\overline{x}^2, \overline{x}) \mid \overline{x} \in {\OO_{S_{n}}^*}/{\OO_{S_{n}}^{*3}} \}$,\end{small} where \begin{small}$\OO_{S_{n}}^*$\end{small} denotes the set of \begin{small}$S_n$\end{small}-units in \begin{small}$K$,\end{small} has \begin{small}$\F_3$-\end{small}dimension \begin{small}$|S_n|+1$\end{small} by Dirichlet's $S$-units theorem. In particular, \begin{small}$(\overline{\zeta^2},\overline{\zeta}) \in \big({A_{S_{n}}^*}/{A_{S_{n}}^{*3}}\big)_{N=1}$.\end{small} 
On the other hand, note from Lemma \ref{imgofdeltaatp} that \begin{small}$(\overline{\zeta^2},\overline{\zeta}) \notin \text{Im } \delta_{\phi,K_\p}$,\end{small} whenever \begin{small}$n \not\equiv \pm 1 \pmod 9$.\end{small} Consequently, \begin{small}$(\overline{\zeta^2},\overline{\zeta}) \notin {\rm Sel}^\phi (E_{16n^2}/K)$\end{small}, so \begin{small}${\rm Sel}^\phi (E_{16n^2}/K) \subsetneq \big({A_{S_{n}}^*}/{A_{S_{n}}^{*3}}\big)_{N=1}$.\end{small} Thus, one obtains \begin{small}$\dim_{\F_3} {\rm Sel}^\phi (E_{16n^2}/K) \lneq \dim_{\F_3} \big({A_{S_{n}}^*}/{A_{S_{n}}^{*3}}\big)_{N=1} = 2k_1+k_2+1 $.\end{small}
\end{proof}

\begin{lemma}\label{rankbound}
    Let \begin{small}$E_{16n^2}:y^2=x^3+16n^2$\end{small} be an elliptic curve and \begin{small}$\phi:E_{16n^2} \to E_{16n^2}$\end{small} the $3$-isogeny over \begin{small}$K$\end{small} defined in \eqref{eq:defofphi}. If \begin{small} $ \dim_{\F_3} {\rm Sel}^\phi (E_{16n^2}/K)=t$,\end{small} then \begin{small} $ \text{rk } E_{16n^2}(\Q) \le t-1$.\end{small}\\
    Moreover, if \begin{small} $ \dim_{\F_3} \Sh(E_{16n^2}/\Q)[3]$\end{small} is even, then  \begin{small} $ \text{rk } E_{16n^2}(\Q) \equiv t-1 \pmod 2$.\end{small}
\end{lemma}
\begin{proof}
    We get \begin{small}$\text{rk } E_{16n^2}(\Q) \le \dim_{\F_3} {\rm Sel}^\phi (E_{16n^2}/K) - 1$\end{small} from \eqref{eq:expforrk}. So if \begin{small}$ \dim_{\F_3} {\rm Sel}^\phi (E_{16n^2}/K)=t$,\end{small} then  \begin{small}$\text{rk } E_{16n^2}(\Q) \le t-1$.\end{small} 

    Recall that \begin{small}$\dim_{\F_3} {\rm Sel}^3({E}_{16n^2}/\Q) = \text{rk } E_{16n^2}(\Q) + \dim_{\F_3} E_{16n^2}(\Q)[3] + \dim_{\F_3} \Sh(E_{16n^2}/\Q)[3] = \text{rk } E_{16n^2}(\Q) + \dim_{\F_3} \Sh(E_{16n^2}/\Q)[3] +1$, \end{small} since \begin{small} $|E_{16n^2}(\Q)[3]|=3$.\end{small}
    
    On the other hand, consider the rational $3$-isogeny \begin{small}$\phi_n:E_{16n^2} \to E_{-432n^2}$\end{small} defined in the Introduction and its dual isogeny \begin{small}$\widehat{\phi}_n:E_{-432n^2} \to E_{16n^2}$\end{small}. 
    Note that \begin{small}$|E_{-432n^2}(\Q)[\widehat{\phi}_n]|=1$.\end{small} Set \begin{small}$R:=\frac{\Sh(E_{-432n^2}/\Q)[\widehat{\phi}_n]}{\phi_n(\Sh(E_{16n^2}/\Q)[3])}$\end{small}. Then by \cite[Prop.~49]{bes} we have the following exact sequence
    \begin{small}$$0 \rightarrow \frac{E_{-432n^2}(\Q)[\widehat{\phi}_n]}{{\phi}_n({E}_{16n^2}(\Q)[3])} \rightarrow {\rm Sel}^{{\phi}_n}({E}_{16n^2}/\Q) \rightarrow {\rm Sel}^3({E}_{16n^2}/\Q) \rightarrow {\rm Sel}^{\widehat{\phi}_n}({E}_{-432n^2}/\Q) \rightarrow R \rightarrow 0.$$\end{small}
\noindent So, \begin{small}$\dim_{\F_3} {\rm Sel}^3({E}_{16n^2}/\Q) = \dim_{\F_3} {\rm Sel}^{{\phi}_n}({E}_{16n^2}/\Q) + \dim_{\F_3} {\rm Sel}^{\widehat{\phi}_n}({E}_{-432n^2}/\Q) - \dim_{\F_3} \frac{E_{-432n^2}(\Q)[\widehat{\phi}_n]}{{\phi}_n({E}_{16n^2}(\Q)[3])}$ $- \dim_{\F_3} R = \dim_{\F_3} {\rm Sel}^{{\phi}_n}({E}_{16n^2}/\Q) + \dim_{\F_3} {\rm Sel}^{\widehat{\phi}_n}({E}_{-432n^2}/\Q) - \dim_{\F_3} R $,\end{small} since \begin{small}$E_{-432n^2}(\Q)$\end{small} has no non-trivial $3$-torsion points. Now, we have \begin{small}${\rm Sel}^\phi(E_{16n^2}/K) \cong {\rm Sel}^{\phi_n}(E_{16n^2}/\Q) \oplus {\rm Sel}^{\widehat{\phi}_n}(E_{-432n^2}/\Q)$\end{small} by \cite[Lemma~3.20]{jms}. Therefore, we obtain 
\begin{small}$\dim_{\F_3} {\rm Sel}^3({E}_{16n^2}/\Q) =$ $\dim_{\F_3} {\rm Sel}^{{\phi}}({E}_{16n^2}/K)- \dim_{\F_3} R =$ $ t  - \dim_{\F_3} R$\end{small}. 

Then the two expressions for \begin{small}$\dim_{\F_3} {\rm Sel}^3(E_{16n^2}/\Q)$\end{small}  composedly imply \begin{small}$$\text{rk } E_{16n^2}(\Q)=\dim_{\F_3} {\rm Sel}^{{\phi}}({E}_{16n^2}/K) - \dim_{\F_3} \Sh(E_{16n^2}/\Q)[3]   - \dim_{\F_3} R -1.$$\end{small}
We know that \begin{small}$\dim_{\F_3} R$\end{small} is even by \cite[Prop.~49]{bes}. Now if we assume \begin{small}$\dim_{\F_3} \Sh(E_{16n^2}/\Q)[3]$\end{small} is even, then \begin{small}$\text{rk } E_{16n^2}(\Q) \equiv \dim_{\F_3} {\rm Sel}^{{\phi}}({E}_{16n^2}/K) -1 = t - 1 \pmod 2.$\end{small}
\end{proof}

\section{Explicit computation of $3$-rank of ${\rm Sel}^\phi(E_{16n^2}/K)$}\label{sec2}
Henceforth, we restrict our attention to cube-free natural numbers \begin{small}$n$\end{small} that are co-prime to $3$ and have exactly two distinct prime factors, that is \begin{small}$n$\end{small} is of the form \begin{small}$n= \ell_1\ell_2$\end{small} or \begin{small}$n= \ell_1^2\ell_2$\end{small} or \begin{small}$n= \ell_1^2\ell_2^2$,\end{small} where \begin{small}$\ell_1 \neq 3$\end{small} and \begin{small}$\ell_2 \neq 3$\end{small} are two distinct primes. In the language of the Generalised Sylvester's Conjecture, we are studying the case \begin{small}$k_1+k_2=2$.\end{small}
We consider the cases \begin{small}$k_1=0$, $k_1=1$\end{small} and \begin{small}$k_1=2$\end{small} separately in three different subsections. In each of the three subsections, we obtain the \begin{small}$\F_3$\end{small}-dimensions of \begin{small}${\rm Sel}^{{\phi}}(E_{16n^2}/K)$\end{small} for integers \begin{small}$n \in \{\ell_1\ell_2, \ell_1^2\ell_2, \ell_1\ell_2^2, \ell_1^2\ell_2^2\}$,\end{small} 
where \begin{small}$\ell_1$\end{small} and \begin{small}$\ell_2$\end{small} are two distinct primes both different from $3$. However, we furnish the proofs only for \begin{small}$n=\ell_1\ell_2$\end{small} to avoid redundancy of literature, as the proofs for other types of \begin{small}$n$\end{small} are similar in nature.

\subsection{The Case $k_1=0$, i.e., $\ell_1 \equiv \ell_2 \equiv 2 \pmod 3$}
In this subsection, we assume that \begin{small}$\ell_1 \equiv \ell_2 \equiv 2 \pmod 3$.\end{small} We have the following results regarding \begin{small}$\dim_{\F_3}{\rm Sel}^{{\phi}}({E}_{16n^2}/K)$:\end{small}

\begin{proposition}\label{sylvesterpq}
Let \begin{small}$\ell_1 \equiv \ell_2 \equiv 2 \pmod 3$\end{small} be distinct primes and \begin{small}$n$\end{small} is either \begin{small}$\ell_1\ell_2$\end{small} or \begin{small}$\ell_1^2\ell_2^2$\end{small}. Let \begin{small}$\phi: E_{16n^2} \to E_{16n^2}$\end{small} be the $3$-isogeny over \begin{small}$K$\end{small} given in \eqref{eq:defofphi}. Then \begin{small}$$\dim_{\F_3} {\rm Sel}^\phi(E_{16n^2}/K) =\begin{cases} 1, & \text{ if } \ell_1 \equiv 2 \pmod 9 \text{ and } \ell_2 \equiv 5 \pmod 9,\\ 3, & \text{ if }  \ell_1 \equiv \ell_2 \equiv 8 \pmod 9,\\ 2, & \text{ if }  n \equiv 4,7 \pmod 9.\end{cases}$$\end{small}
\end{proposition}

\begin{proof}
  We prove the result for \begin{small}$n = \ell_1\ell_2$.\end{small} Observe that \begin{small}$S_{n}=\{\ell_1\OO_K, \ell_2\OO_K\}$.\end{small} By Lemma \ref{selmerbound}, we have \begin{small} ${\rm Sel}^\phi(E_{16n^2}/K) \subset \Big(\frac{A_{S_{n}}^*}{A_{S_{n}}^{*3}}\Big)_{N=1}$\end{small} and \begin{small}$\dim_{\F_3} {\rm Sel}^\phi(E_{16n^2}/K) \le |S_n|+1= 3$.\end{small}
The group of \begin{small}$S_{n}$\end{small}-units in \begin{small}$K$\end{small} is \begin{small}$\OO^*_{K,S_{n}}=\langle \pm \zeta, \ell_1, \ell_2 \rangle=\langle \pm \zeta, \ell_1, n \rangle$.\end{small} Hence, \begin{small} $\Big(\frac{A_{S_{n}}^*}{A_{S_{n}}^{*3}}\Big)_{N=1} = \big\langle\big(\overline{\zeta^2},\overline{\zeta}\big), \big(\overline{\ell_1^2},\overline{\ell_1}\big), \big(\overline{n}^2,\overline{n}\big)\big\rangle$.\end{small} 
Note that \begin{small}$(0,4n) \in E_{16n^2}(\Q)$\end{small} and  \begin{small}$\delta_{\phi,K_\q}(0,4n)= \big(\overline{n}^2,\overline{n}\big)$\end{small} for all \begin{small}$\q$.\end{small} Hence, \begin{small}$\big(\overline{n}^2,\overline{n}\big) \in {\rm Sel}^\phi(E_{16n^2}/K)$.\end{small} Thus, \begin{small}$1 \le \dim_{\F_3} {\rm Sel}^\phi(E_{16n^2}/K) \le 3$.\end{small} We now examine the remaining generators of \begin{small}$\Big(\frac{A_{S_{n}}^*}{A_{S_{n}}^{*3}}\Big)_{N=1}$\end{small} for their containment in \begin{small} ${\rm Sel}^\phi(E_{16n^2}/K)$.\end{small}

First, we discuss the generator \begin{small}$(\overline{\zeta}^2,\overline{\zeta})$\end{small}. For \begin{small}$\ell_1 \equiv \ell_2 \equiv 8 \pmod 9,$\end{small} we know that \begin{small}$\zeta_9 \in K_{\ell_1}$\end{small} and \begin{small}$\zeta_9 \in K_{\ell_1}$\end{small}. Then it follows that \begin{small}$(\overline{\zeta}^2,\overline{\zeta})=(\overline{1},\overline{1}) \in \big(A_{\ell_i}^*/A_{\ell_i}^{*3}\big)_{N=1}$\end{small} for \begin{small}$i=1,2$.\end{small} Hence, \begin{small}$(\overline{\zeta}^2,\overline{\zeta})$\end{small} belongs to both \begin{small}$\text{Im } \delta_{\phi,K_{\ell_1}}$\end{small} and \begin{small}$\text{Im } \delta_{\phi,K_{\ell_2}}$\end{small} by Lemma \ref{imgofkummap}. Also  \begin{small}$(\overline{\zeta}^2,\overline{\zeta}) \in \text{Im } \delta_{\phi,K_{\p}}$\end{small} by Lemma \ref{imgofdeltaatp}(2) in this situation. For primes \begin{small}$\q \notin S_{n} \cup \{\p\}$, $(\overline{\zeta}^2,\overline{\zeta}) \in \big({A_\q^*}/{A_\q^{*3}}\big)_{N=1}=\text{Im } \delta_{\phi,K_\q}$\end{small} by Lemma \ref{imgofkummap}. So \begin{small}$(\overline{\zeta}^2,\overline{\zeta}) \in {\rm Sel}^\phi(E_{16n^2}/K)$\end{small}, when \begin{small}$\ell_1 \equiv \ell_2 \equiv 8 \pmod 9$.\end{small} On the other hand, if \begin{small}$\ell_1 \equiv 2 \text{ or } 5 \pmod 9$,\end{small} then we know that \begin{small}$\zeta_9 \notin K_{\ell_1}$\end{small} and \begin{small}$\text{Im } \delta_{\phi,K_{\ell_1}} \cap \big(A_{\ell_1}^*/A_{\ell_1}^{*3}\big)_{N=1}=\{(\overline{1},\overline{1})\}$\end{small} by Lemma \ref{imgofkummap}. Hence, \begin{small}$(\overline{\zeta}^2,\overline{\zeta}) \notin \text{Im } \delta_{\phi,K_{\ell_1}}$\end{small} and so \begin{small}$(\overline{\zeta}^2,\overline{\zeta}) \notin {\rm Sel}^\phi(E_{16n^2}/K)$\end{small} when \begin{small}$\ell_1 \equiv 2 \text{ or } 5 \pmod 9$.\end{small}\\

We now investigate the remaining generator \begin{small}$(\overline{\ell_1^2},\overline{\ell_1})$\end{small} of \begin{small} $\Big(\frac{A_{S_{n}}^*}{A_{S_{n}}^{*3}}\Big)_{N=1}$.\end{small} Note that for primes \begin{small}$\q \notin S_{n} \cup \{\p\}$,\end{small} one has  \begin{small}$(\overline{\ell_1^2},\overline{\ell_1}) \in \text{Im } \delta_{\phi,K_\q}$\end{small} by Lemma \ref{imgofkummap}. Thus, we only need to check whether  \begin{small}$(\overline{\ell_1^2},\overline{\ell_1}) \in \text{Im } \delta_{\phi,K_\q}$\end{small} for \begin{small}$\q \in S_{n} \cup \{\p\}$.\end{small} Note that the cubic residue \begin{small}$\big(\frac{\ell_1}{\ell_2}\big)_3=1$,\end{small} as both \begin{small}$\ell_1$\end{small} and \begin{small}$ \ell_2$\end{small} are primes congruent to \begin{small}$2 \pmod 3.$\end{small} Hence, \begin{small}$(\overline{\ell_1^2},\overline{\ell_1})=(\overline{1},\overline{1}) \in \big({A_{\ell_2}^*}/{A_{\ell_2}^{*3}}\big)_{N=1}$\end{small} and so  \begin{small}$(\overline{\ell_1^2},\overline{\ell_1}) \in \text{Im } \delta_{\phi,K_{\ell_2}}$.\end{small} Similarly using \begin{small}$\big(\frac{\ell_2}{\ell_1}\big)_3=1$,\end{small} we get \begin{small}$(\overline{\ell_1^2},\overline{\ell_1})=(\overline{n}^2,\overline{n}) \in \text{Im } \delta_{\phi,K_{\ell_1}}$.\end{small} Thus, it reduces to check whether \begin{small}$(\overline{\ell_1^2},\overline{\ell_1}) \in \text{Im } \delta_{\phi,K_\p}$.\end{small}

Let us first take \begin{small}$\ell_1 \equiv \ell_2 \equiv 8 \pmod 9$.\end{small} In this case, \begin{small}$\ell_1^2 \equiv 1 \pmod 9$\end{small} i.e. \begin{small}$\ell_1^2 \in U^4_{K_\p}$.\end{small} Since \begin{small}$\frac{U^3_{K_\p}}{U^4_{K_\p}} \cong \frac{V_3}{A_\p^{*3}} \subset \text{Im } \delta_{\phi,K_\p}$\end{small} by Lemmas \ref{V3U1U4} and \ref{imgofkummap}, it follows that \begin{small}$(\overline{\ell_1^2},\overline{\ell_1})=(\overline{1},\overline{1}) \in \text{Im } \delta_{\phi,K_\p}$.\end{small} Hence, \begin{small}$(\overline{\ell_1^2},\overline{\ell_1}) \in {\rm Sel}^\phi(E_{16n^2}/K)$\end{small} and \begin{small}$\dim_{\F_3} {\rm Sel}^\phi(E_{16n^2}/K) =3$\end{small} when \begin{small}$\ell_1 \equiv \ell_2 \equiv 8 \pmod 9$.\end{small}  

 Next, let \begin{small}$n=\ell_1\ell_2 \not\equiv 1 \pmod 9$\end{small} and 
notice that \begin{small}$\ell_1^2 \equiv 1 \pmod 3$\end{small} i.e.
\begin{small}$\ell_1^2 \in U^2_{K_\p}$.\end{small} 
Hence by Lemma \ref{imgofdeltaatp}(1), \begin{small}$(\overline{\ell_1^2},\overline{\ell_1}) \in \text{Im } \delta_{\phi,K_\p}$\end{small} and it follows that  \begin{small}$(\overline{\ell_1^2},\overline{\ell_1}) \in {\rm Sel}^\phi(E_{16n^2}/K)$.\end{small} So, \begin{small}$\dim_{\F_3} {\rm Sel}^\phi(E_{16n^2}/K) =2$\end{small} whenever \begin{small}$n=\ell_1\ell_2 \not\equiv 1 \pmod 9$.\end{small}
 
Finally for \begin{small}$\ell_1 \equiv 2 \pmod 9$\end{small} and \begin{small}$\ell_2 \equiv 5 \pmod 9$,\end{small} 
 observe that \begin{small}$\ell_1^2 \in U^2_{K_\p} \setminus U^3_{K_\p}$.\end{small} This implies by Lemma \ref{imgofdeltaatp}(2) that \begin{small}$(\overline{\ell_1^2},\overline{\ell_1}) \notin \text{Im } \delta_{\phi,K_\p}$\end{small} and hence \begin{small}$(\overline{\ell_1^2},\overline{\ell_1}) \notin {\rm Sel}^\phi(E_{16n^2}/K)$.\end{small} There are four subgroups of \begin{small} $\Big(\frac{A_{S_{n}}^*}{A_{S_{n}}^{*3}}\Big)_{N=1} \Big/ \big\langle(\overline{n}^2,\overline{n})\big\rangle$\end{small} of order $3$  corresponding to the elements \begin{small}$\ell_1$, $\zeta$, $\ell_1\zeta$\end{small}  and \begin{small}$\ell_1\zeta^2$.\end{small} To conclude that \begin{small}$\dim_{\F_3} {\rm Sel}^\phi(E_{16n^2}/K) = 1$,\end{small} we need to show that neither
\begin{small}$\big\langle (\overline{\ell_1^2\zeta^2},\overline{\ell_1\zeta}) \big\rangle$\end{small} nor \begin{small}$\big\langle (\overline{\ell_1\zeta^2},\overline{\ell_1^2\zeta}) \big\rangle$\end{small} are contained in \begin{small}${\rm Sel}^\phi(E_{16n^2}/K)$.\end{small} Recall that \begin{small}$\text{Im } \delta_{\phi,K_{\ell_2}} \cap \big({A_{\ell_2}^*}/{A_{\ell_2}^{*3}}\big)_{N=1} = \{(\overline{1},\overline{1})\}$\end{small} by Lemma \ref{imgofkummap}. Using  \begin{small}$\big(\frac{\ell_1}{\ell_2}\big)_3=1$,\end{small} we see that  \begin{small}$(\overline{\ell_1^2\zeta^2},\overline{\ell_1\zeta})=(\overline{\zeta^2},\overline{\zeta}) \in \big({A_{\ell_2}^*}/{A_{\ell_2}^{*3}}\big)_{N=1}$.\end{small} 
Since \begin{small}$\zeta_9 \notin \OO_{K_{\ell_2}}^*$,\end{small} we have \begin{small}$(\overline{\ell_1^2\zeta^2},\overline{\ell_1\zeta})= (\overline{\zeta^2},\overline{\zeta}) \notin \text{Im } \delta_{\phi,K_{\ell_2}}$\end{small} and so it is not in \begin{small}${\rm Sel}^\phi(E_{16n^2}/K)$.\end{small} A similar argument works for \begin{small}$(\overline{\ell_1^2\zeta},\overline{\ell_1\zeta^2})$.\end{small} This shows that \begin{small}$\dim_{\F_3} {\rm Sel}^\phi(E_{16n^2}/K) =1$,\end{small} when \begin{small}$\ell_1 \equiv 2 \pmod 9$\end{small} and \begin{small}$\ell_2 \equiv 5 \pmod 9$.\end{small} 
This completes the proof for \begin{small}$n=\ell_1\ell_2$.\end{small}
\end{proof}

\begin{proposition}\label{sylvesterp2q}
Let \begin{small}$\ell_1 \equiv \ell_2 \equiv 2 \pmod 3$\end{small} be distinct primes and \begin{small}$n= \ell_1^2\ell_2$\end{small}. For the $3$-isogeny \begin{small}$\phi: E_{16n^2} \to E_{16n^2}$\end{small} over \begin{small}$K$\end{small} given in \eqref{eq:defofphi}, we have \begin{small}$$\dim_{\F_3} {\rm Sel}^\phi(E_{16n^2}/K) =\begin{cases} 1, & \text{ if } \ell_1 \equiv \ell_2 \equiv 2 \text{ or } 5 \pmod 9 ,\\ 
3, & \text{ if }  \ell_1 \equiv \ell_2 \equiv 8 \pmod 9,\\
2, & \text{ if }  n \equiv 2,5 \pmod 9. \hspace{2cm} \qed\end{cases}$$\end{small}
\end{proposition}

\begin{rem}
    The results in this subsection easily generalize to any cube-free integer \begin{small}$n$,\end{small} co-prime to $3$, all of whose prime divisors \begin{small}$\ell_1, \dots, \ell_k$\end{small} are congruent to $2$ modulo $3$. In this situation we have,
    \begin{small}$$\dim_{\F_3} {\rm Sel}^\phi(E_{16n^2}/K) =\begin{cases} k, & \text{ if } n \not \equiv \pm 1 \pmod 9 ,\\ 
k+1, & \text{ if }  \ell_1 \equiv \ell_2 \equiv \cdots \equiv \ell_k \equiv 8 \pmod 9,\\
k-1, & \text{ otherwise. }\end{cases}$$\end{small}
To see this, observe that if \begin{small}$n \not \equiv \pm 1 \pmod 9$,\end{small} then \begin{small}$(\overline{\zeta^2},\overline{\zeta}) \notin {\rm Sel}^\phi(E_{16n^2}/K)$,\end{small} whereas  \begin{small}$(\overline{\ell_i^2},\overline{\ell_i}) \in {\rm Sel}^\phi(E_{16n^2}/K)$\end{small} for each \begin{small}$1 \le i \le k$.\end{small} If \begin{small}$\ell_1 \equiv \ell_2 \equiv \cdots \equiv \ell_k \equiv 8 \pmod 9$,\end{small} then \begin{small}$(\overline{\zeta^2},\overline{\zeta}) \in {\rm Sel}^\phi(E_{16n^2}/K)$\end{small} and for every \begin{small}$1 \le i \le k$, $(\overline{\ell_i^2},\overline{\ell_i}) \in {\rm Sel}^\phi(E_{16n^2}/K)$.\end{small} Finally, if \begin{small}$n \equiv \pm 1 \pmod 9$\end{small} and if some of the \begin{small}$\ell_i \not\equiv 8 \pmod 9$,\end{small} writing \begin{small}$m = \zeta^{a_0} \ell_1^{a_1} \cdots \ell_k^{a_k}$,\end{small} one checks that \begin{small}$(\overline{m}^2,\overline{m}) \in {\rm Sel}^\phi(E_{16n^2}/K)$\end{small} if and only if \begin{small}$a_0=0$\end{small} and \begin{small}$m \equiv \pm 1 \pmod 9$.\end{small} An easy combinatorial argument shows that there are \begin{small}$3^{k-1}$\end{small} such elements.
\end{rem}

\subsection{The Case $k_1=1$, that is, $\ell_1 \equiv 1 \pmod 3$ and  $\ell_2 \equiv 2 \pmod 3$}
In this subsection, we study the situation when one of the prime divisors  of \begin{small}$n$\end{small} is congruent to \begin{small}$1 \pmod 3$,\end{small} while its other prime divisor is congruent to \begin{small}$2 \pmod 3$.\end{small} For simplicity of notation, we assume that \begin{small}$\ell_1 \equiv 1 \pmod 3$\end{small} and  \begin{small}$\ell_2 \equiv 2 \pmod 3$\end{small}.  
Note that \begin{small}$\ell_1$\end{small} splits as \begin{small}$\ell_1=\pi_{\ell_1}\pi'_{\ell_1}$\end{small} in \begin{small}$K$\end{small}, while \begin{small}$\ell_2$\end{small} remains inert in \begin{small}$K$.\end{small} We denote by \begin{small}$\big(\frac{\cdot}{\cdot}\big)_3$\end{small} the cubic residue symbol. Note that \begin{small}$\big(\frac{\ell_2}{\pi_{\ell_1}}\big)_3 = 1$\end{small} if and only if \begin{small}$\ell_2$\end{small} is a cube in \begin{small}$\F_{\ell_1}$.\end{small}

\begin{proposition}\label{sylvesterpq2}
Let \begin{small}$n$\end{small} be an integer of the form either \begin{small}$\ell_1\ell_2$\end{small} or \begin{small}$\ell_1^2\ell_2$\end{small} or \begin{small}$\ell_1\ell_2^2$\end{small} or \begin{small}$\ell_1^2\ell_2^2$\end{small}, where \begin{small}$\ell_1 \equiv 1 \pmod 3$\end{small} and \begin{small}$\ell_2 \equiv 2 \pmod 3$\end{small} are primes such that \begin{small}$n \not\equiv \pm 1 \pmod 9$.\end{small} For the $3$-isogeny \begin{small}$\phi: E_{16n^2} \to E_{16n^2}$\end{small} defined in \eqref{eq:defofphi}, we have \begin{small}$$\dim_{\F_3} {\rm Sel}^\phi(E_{16n^2}/K) =\begin{cases} 3,& \text{ if } \big(\frac{\ell_2}{\pi_{\ell_1}}\big)_3 = 1, \\ 1, & \text{ otherwise.}\end{cases}$$\end{small}
\end{proposition}

\proof
     We prove the result for \begin{small}$n= \ell_1\ell_2$\end{small} and mention that the proof in other cases for \begin{small}$n$\end{small} given above can be obtained via a similar computation. We have \begin{small}$S_{n}=\{\pi_{\ell_1}, \pi'_{\ell_1}, \ell_2\OO_K\}$\end{small} and by Lemma \ref{selmerbound}, \begin{small} ${\rm Sel}^\phi(E_{16n^2}/K) \subset \Big(\frac{A_{S_{n}}^*}{A_{S_{n}}^{*3}}\Big)_{N=1}$\end{small} and \begin{small}$\dim_{\F_3} {\rm Sel}^\phi(E_{16n^2}/K) \le |S_n|+1= 4$.\end{small}
 The group of \begin{small}$S_{n}$\end{small}-units in \begin{small}$K$\end{small} is \begin{small}$\OO^*_{K,S_{n}}=\langle \pm \zeta, \pi_{\ell_1}, \pi'_{\ell_1}, \ell_2 \rangle=\langle \pm \zeta, \pi_{\ell_1}, \ell_2, n \rangle$.\end{small} Hence, \begin{small} $\Big(\frac{A_{S_{n}}^*}{A_{S_{n}}^{*3}}\Big)_{N=1} = \big\langle(\overline{\zeta^2},\overline{\zeta}), (\overline{\ell_2^2},\overline{\ell_2}),  (\overline{\pi_{\ell_1}}^2,\overline{\pi_{\ell_1}}), (\overline{n}^2,\overline{n})\big\rangle$.\end{small} 
 Note that \begin{small}$(0,4n) \in E_{16n^2}(\Q)$\end{small} and  \begin{small}$\delta_{\phi,K_\q}(0,4n)=(\overline{n}^2,\overline{n}),$\end{small} \begin{small}$\forall \q$\end{small} by the Definition \eqref{eq:kummap} of the local Kummer maps. Hence, \begin{small}$(\overline{n}^2,\overline{n}) \in {\rm Sel}^\phi(E_{16n^2}/K)$.\end{small} 

  If \begin{small}$\ell_2 \equiv 2 \text{ or } 5 \pmod 9$,\end{small} then we know that \begin{small}$\zeta_9 \notin K_{\ell_2}$\end{small} and  \begin{small}$\text{Im } \delta_{\phi,K_{\ell_2}} \cap \big(A_{\ell_2}^*/A_{\ell_2}^{*3}\big)_{N=1}=\{(\overline{1},\overline{1})\}$\end{small} by Lemma \ref{imgofkummap}. Hence, \begin{small}$(\overline{\zeta^2},\overline{\zeta}) \notin \text{Im } \delta_{\phi,K_{\ell_2}}$.\end{small}  On the other hand, if  \begin{small}$\ell_2 \equiv 8 \pmod 9,$\end{small} then \begin{small}$\ell_1 \equiv 4 \text{ or } 7 \pmod 9$.\end{small} So by a similar argument \begin{small}$(\overline{\zeta^2},\overline{\zeta}) \notin \text{Im } \delta_{\phi,K_{\pi_{\ell_1}}}$.\end{small}  Thus in either cases, we get that \begin{small}$(\overline{\zeta^2},\overline{\zeta}) \notin {\rm Sel}^\phi(E_{16n^2}/K)$\end{small} and so \begin{small}$1 \le \dim_{\F_3} {\rm Sel}^\phi(E_{16n^2}/K) \le 3$.\end{small} 

  We now investigate the remaining generators \begin{small}$(\overline{\ell_2^2},\overline{\ell_2}) $\end{small} and \begin{small}$(\overline{\pi_{\ell_1}}^2,\overline{\pi_{\ell_1}})$\end{small} case by case. Recall that one can write \begin{small}$\ell_1=\frac{1}{4}(L^2+27M^2)$\end{small} and  \begin{small}$\pi_{\ell_1}=\frac{1}{2}(L+3\sqrt{-3}M)$\end{small}.
  
\begin{enumerate}
 \item Let us first consider the case \begin{small}$\big(\frac{\ell_2}{\pi_{\ell_1}}\big)_3= 1$\end{small} and 
  investigate the generator \begin{small}$(\overline{\ell_2^2},\overline{\ell_2})$\end{small} of \begin{small} $\Big(\frac{A_{S_{n}}^*}{A_{S_{n}}^{*3}}\Big)_{N=1}$.\end{small} Note that for primes \begin{small}$\q \notin S_{n} \cup \{\p\}$, $(\overline{\ell_2^2},\overline{\ell_2}) \in \text{Im } \delta_{\phi,K_\q}$\end{small} due to Lemma \ref{imgofkummap}. Thus, we only need to check whether or not \begin{small}$(\overline{\ell_2^2},\overline{\ell_2}) \in \text{Im } \delta_{\phi,K_\q}$\end{small} for \begin{small}$\q \in S_{n} \cup \{\p\}$.\end{small} Notice that the cubic residue \begin{small}$\big(\frac{\ell_2}{\pi_{\ell_1}}\big)_3=1$\end{small} implies \begin{small}$(\overline{\ell_2^2},\overline{\ell_2})=(\overline{1},\overline{1}) \in \big({A_{\pi_{\ell_1}}^*}/{A_{\pi_{\ell_1}}^{*3}}\big)_{N=1}$.\end{small} So, it is in \begin{small}$\text{Im } \delta_{\phi,K_{\pi_{\ell_1}}}$.\end{small} Similarly, \begin{small}$\big(\frac{\ell_2}{\pi'_{\ell_1}}\big)_3=1$\end{small} gives \begin{small}$(\overline{\ell_2^2},\overline{\ell_2}) \in \text{Im } \delta_{\phi,K_{\pi'_{\ell_1}}}$.\end{small} Now using \begin{small}$\big(\frac{\ell_1}{{\ell_2}}\big)_3=1$,\end{small} one gets \begin{small}$(\overline{\ell_2^2},\overline{\ell_2})=(\overline{n}^2,\overline{n}) \in \text{Im } \delta_{\phi,K_{\ell_2}}$.\end{small} Finally, \begin{small}$\ell_2 \equiv 2 \pmod 3$\end{small} implies \begin{small}$\ell_2^2 \in U^2_{K_\p}$\end{small}. So by Lemma \ref{imgofdeltaatp}(1), \begin{small}$(\overline{\ell_2^2},\overline{\ell_2}) \in \text{Im } \delta_{\phi,K_\p}$\end{small} and hence \begin{small}$(\overline{\ell_2^2},\overline{\ell_2}) \in {\rm Sel}^\phi(E_{16n^2}/K).$\end{small}

  Next, we draw our attention to the generator \begin{small}$(\overline{\pi_{\ell_1}}^2,\overline{\pi_{\ell_1}})$\end{small} of \begin{small} $\Big(\frac{A_{S_{n}}^*}{A_{S_{n}}^{*3}}\Big)_{N=1}$\end{small} and note that the cubic residue \begin{small}$\big(\frac{\pi_{\ell_1}}{\pi'_{\ell_1}}\big)_3=1$\end{small} implies \begin{small}$(\overline{\pi_{\ell_1}}^2,\overline{\pi_{\ell_1}})=(\overline{1},\overline{1}) \in \big({A_{\pi'_{\ell_1}}^*}/{A_{\pi'_{\ell_1}}^{*3}}\big)_{N=1}$.\end{small} This gives us   \begin{small}$(\overline{\pi_{\ell_1}}^2,\overline{\pi_{\ell_1}}) \in \text{Im } \delta_{\phi,K_{\pi'_{\ell_1}}}$.\end{small} Similarly, \begin{small}$\big(\frac{\pi_{\ell_1}}{\ell_2}\big)_3=1$\end{small} gives  \begin{small}$(\overline{\pi_{\ell_1}}^2,\overline{\pi_{\ell_1}}) \in \text{Im } \delta_{\phi,K_{{\ell_2}}}$.\end{small} Additionally,  one has \begin{small}$(\overline{\pi_{\ell_1}}^2,\overline{\pi_{\ell_1}})=(\overline{n}^2,\overline{n}) \in \text{Im } \delta_{\phi,K_{\pi_{\ell_1}}}$\end{small} due to the assumption \begin{small}$\big(\frac{\ell_2}{\pi_{\ell_1}}\big)_3= 1$.\end{small} Finally, writing \begin{small}$\pi_{\ell_1}=\frac{1}{2}(L+3\sqrt{-3}M)$\end{small}, we have \begin{small}$\pi_{\ell_1}^2 \in U^2_{K_\p}$\end{small} which implies \begin{small}$(\overline{\pi_{\ell_1}}^2,\overline{\pi_{\ell_1}}) \in \text{Im } \delta_{\phi,K_\p}$\end{small} by Lemma \ref{imgofdeltaatp}(1).  Hence, \begin{small}$(\overline{\pi_{\ell_1}}^2,\overline{\pi_{\ell_1}}) \in {\rm Sel}^\phi(E_{16n^2}/K)$\end{small} and we get \begin{small}$\dim_{\F_3}{\rm Sel}^\phi(E_{16n^2}/K)=3$\end{small}  when \begin{small}$n=\ell_1\ell_2 \not\equiv 8 \pmod 9$\end{small} and \begin{small}$\big(\frac{\ell_2}{\pi_{\ell_1}}\big)_3= 1$.\end{small}\\
  
 \item We now consider the case \begin{small}$\big(\frac{\ell_2}{\pi_{\ell_1}}\big)_3 \neq 1$.\end{small} There are $13$ subgroups of order $3$ of \begin{small} $\Big(\frac{A_{S_{n_{ }}}^*}{A_{S_{n}}^{*3^{ }}}\Big)_{N=1} \Big/ \big\langle(\overline{n}^2,\overline{n})\big\rangle$\end{small}  corresponding to the elements: \begin{small}$\zeta$, $\ell_2$, $ \pi_{\ell_1},$\end{small} \begin{small}$ \zeta\ell_2 $, $ \zeta\ell_2^2$, $ \zeta\pi_{\ell_1}$, $ \zeta\pi_{\ell_1}^2$, $ \ell_2\pi_{\ell_1}$, $ \ell_2^2\pi_{\ell_1}$, $ \zeta\ell_2\pi_{\ell_1},$\end{small}  \begin{small}$ \zeta^2\ell_2\pi_{\ell_1}$, $ \zeta\ell_2^2\pi_{\ell_1}$\end{small} and \begin{small}$\zeta\ell_2\pi_{\ell_1}^2$.\end{small} Thus, to arrive at the conclusion that \begin{small}$\dim_{\F_3}{\rm Sel}^\phi(E_{16n^2}/K)=1$,\end{small} one requires to show that none of these $13$ subgroups is contained in \begin{small}${\rm Sel}^\phi(E_{16n^2}/K)$.\end{small} The fact that \begin{small}$(\overline{\zeta^2},\overline{\zeta}) \notin {\rm Sel}^\phi(E_{16n^2}/K)$ \end{small} is already discussed above. One can easily see that \begin{small}$(\overline{\pi_{\ell_1}}^2,\overline{\pi_{\ell_1}}) \neq (\overline{1},\overline{1}) \in \big(A_{\ell_2}^*/A_{\ell_2}^{*3}\big)_{N=1}$,\end{small} since \begin{small}$\big(\frac{\ell_2}{\pi_{\ell_1}}\big)_3=\big(\frac{\pi_{\ell_1}}{{\ell_2}}\big)_3 \neq 1$\end{small} by assumption. Thus, \begin{small}$(\overline{\pi_{\ell_1}}^2,\overline{\pi_{\ell_1}}) \notin \text{Im } \delta_{\phi,K_{\ell_2}}$.\end{small} Similar arguments using \begin{small}$\big(\frac{\ell_2}{\pi'_{\ell_1}}\big)_3 \neq 1$\end{small} and \begin{small}$\big(\frac{\pi_{\ell_1}}{\pi'_{\ell_1}}\big)_3=1$\end{small} prove that  \begin{small}$(\overline{\ell_2^2},\overline{\ell_2})$, $(\overline{\ell_2^2\pi_{\ell_1}^2},\overline{\ell_2\pi_{\ell_1}})$\end{small} and \begin{small}$(\overline{\ell_2\pi_{\ell_1}^2},\overline{\ell_2^2\pi_{\ell_1}})$\end{small} are all not in \begin{small}$\text{Im } \delta_{\phi,K_{\pi'_{\ell_1}}}$.\end{small} For all the remaining elements, one can see that they are not in \begin{small}${U^2_{K_\p}}/{U^4_{K_\p}}.$\end{small} Hence, they do not belong to \begin{small}$\text{Im } \delta_{\phi,K_{\p}}$\end{small} by Lemma \ref{imgofdeltaatp}(1). Therefore, we deduce that \begin{small}$\dim_{\F_3}{\rm Sel}^\phi(E_{16n^2}/K)=1$. \end{small}
\qed
\end{enumerate}

\begin{proposition}\label{sylvesterpq3}
Let \begin{small}$n$\end{small} be an integer of the form either  \begin{small}$\ell_1\ell_2$\end{small} or \begin{small}$\ell_1^2\ell_2$\end{small} or \begin{small}$\ell_1\ell_2^2$\end{small} or \begin{small}$\ell_1^2\ell_2^2$\end{small}, where \begin{small}$\ell_1 \equiv 1 \pmod 3$\end{small} and \begin{small}$\ell_2 \equiv 2 \pmod 3$\end{small} are primes such that \begin{small}$n \equiv \pm 1 \pmod 9$.\end{small} For the $3$-isogeny \begin{small}$\phi: E_{16n^2} \to E_{16n^2}$\end{small} defined in \eqref{eq:defofphi}, we have \begin{small}$$\dim_{\F_3} {\rm Sel}^\phi(E_{16n^2}/K) =\begin{cases} 4,& \text{ if } \ell_1 \equiv 1 \pmod 9, \enspace \ell_2 \equiv 8 \pmod 9  \text{ and } \big(\frac{\ell_2}{\pi_{\ell_1}}\big)_3 = 1, \\ 2, & \text{ otherwise.}\end{cases}$$\end{small}
\end{proposition}

\begin{proof}
    We prove the result for $n= \ell_1\ell_2$. We know that \begin{small}$S_{n}=\{\pi_{\ell_1}, \pi'_{\ell_1}, \ell_2\OO_K\}$\end{small} and by Lemma \ref{selmerbound}, \begin{small} ${\rm Sel}^\phi(E_{16n^2}/K) \subset \Big(\frac{A_{S_{n}}^*}{A_{S_{n}}^{*3}}\Big)_{N=1}=\big\langle(\overline{\zeta^2},\overline{\zeta}), (\overline{\ell_2^2},\overline{\ell_2}),  (\overline{\pi_{\ell_1}^2},\overline{\pi_{\ell_1}}), (\overline{n}^2,\overline{n})\big\rangle$\end{small} and \begin{small}$\dim_{\F_3} {\rm Sel}^\phi(E_{16n^2}/K) \le |S_n|+1= 4$.\end{small} 
 As in Prop. \ref{sylvesterpq2}, we have  \begin{small}$(0,4n) \in E_{16n^2}(\Q)$\end{small} and  \begin{small}$\delta_{\phi,K_\q}(0,4n)=(\overline{n}^2,\overline{n})$\end{small} for all \begin{small}$\q$.\end{small} Hence, \begin{small}$(\overline{n}^2,\overline{n}) \in {\rm Sel}^\phi(E_{16n^2}/K)$\end{small} and  \begin{small}$1 \le \dim_{\F_3} {\rm Sel}^\phi(E_{16n^2}/K) \le 4$.\end{small}
\begin{enumerate}
    \item Let us first study the case \begin{small}$\ell_1 \equiv 1 \pmod 9$\end{small} and \begin{small}$\ell_2 \equiv 8 \pmod 9$.\end{small} Note that \begin{small}$\zeta_9$\end{small} belongs to each of  \begin{small}$\OO_{K_{\ell_2}}^*$, $\OO_{K_{\pi_{\ell_1}}}^*$\end{small} and  \begin{small}$\OO_{K_{\pi'_{\ell_1}}}^*$.\end{small} Hence, \begin{small}$(\overline{\zeta^2},\overline{\zeta})$\end{small} belongs to each of \begin{small}$\text{Im } \delta_{\phi,K_{\ell_2}}$, $\text{Im } \delta_{\phi,K_{\pi_{\ell_1}}}$\end{small} and \begin{small}$\text{Im } \delta_{\phi,K_{\pi'_{\ell_1}}}$.\end{small} Also by Lemma \ref{imgofdeltaatp}(2), we know that \begin{small}$(\overline{\zeta^2},\overline{\zeta}) \in \text{Im } \delta_{\phi,K_{\p}}$.\end{small} Thus, \begin{small}$(\overline{\zeta^2},\overline{\zeta}) \in {\rm Sel}^\phi(E_{16n^2}/K)$\end{small} in this situation.

    Now if \begin{small}$\big(\frac{\ell_2}{\pi_{\ell_1}}\big)_3=1$\end{small}, then giving a similar argument as in the proof of (1) in Prop. \ref{sylvesterpq2}, we obtain that both  \begin{small}$(\overline{\ell_2^2},\overline{\ell_2})$\end{small} and \begin{small}$(\overline{\pi_{\ell_1}}^2,\overline{\pi_{\ell_1}})$\end{small} belong to \begin{small}${\rm Sel}^\phi(E_{16n^2}/K)$.\end{small} Hence, \begin{small}$\dim_{\F_3} {\rm Sel}^\phi(E_{16n^2}/K)=4$\end{small} when \begin{small}$\ell_1 \equiv 1 \pmod 9$, $\ell_2 \equiv 8 \pmod 9$\end{small} and \begin{small}$\big(\frac{\ell_2}{\pi_{\ell_1}}\big)_3=1.$\end{small}

    Next, let \begin{small}$\big(\frac{\ell_2}{\pi_{\ell_1}}\big)_3 \neq 1 $.\end{small} Then there are four subgroups of \begin{small} $\Big(\frac{A_{S_{n}}^*}{A_{S_{n}}^{*3}}\Big)_{N=1} \Big/ \big\langle(\overline{n}^2,\overline{n}), (\overline{\zeta^2},\overline{\zeta})\big\rangle$\end{small} of order $3$ corresponding to the elements \begin{small}$\ell_2,$ $ \pi_{\ell_1}, $ $\ell_2\pi_{\ell_1}$\end{small} and \begin{small}$\ell_2\pi_{\ell_1}^2$\end{small}. 
    To conclude that \begin{small}$\dim_{\F_3} {\rm Sel}^\phi(E_{16n^2}/K) = 2$,\end{small} we need to show that \begin{small}${\rm Sel}^\phi(E_{16n^2}/K)$\end{small} does not contain any of these subgroups. 
     As \begin{small}$\big(\frac{\ell_2}{\pi_{\ell_1}}\big)_3 \neq 1$\end{small} by assumption, we get that \begin{small}$(\overline{\ell_2^2},\overline{\ell_2}) \notin \text{Im } \delta_{\phi,K_{\pi_{\ell_1}}}$.\end{small} Similarly, \begin{small}$\big(\frac{\ell_2}{\pi_{\ell_1}}\big)_3=\big(\frac{\pi_{\ell_1}}{\ell_2}\big)_3$\end{small} implies \begin{small}$(\overline{\pi_{\ell_1}}^2,\overline{\pi_{\ell_1}}) \notin \text{Im } \delta_{\phi,K_{{\ell_2}}}$.\end{small} Next, using \begin{small}$\big(\frac{\pi_{\ell_1}}{\pi'_{\ell_1}}\big)_3 = 1 $\end{small}  we  identify \begin{small}$(\overline{\ell_2^2\pi_{\ell_1}^2},\overline{\ell_2\pi_{\ell_1}})$\end{small} with \begin{small}$(\overline{{\ell_2}^2},\overline{{\ell_2}})$\end{small} in \begin{small}$\big({A_{\pi'_{\ell_1}}^*}/{A_{\pi'_{\ell_1}}^{*3}}\big)_{N=1}.$\end{small} Then \begin{small}$(\overline{{\ell_2}^2},\overline{{\ell_2}}) \neq (\overline{1},\overline{1})$\end{small} in \begin{small}$\big({A_{\pi'_{\ell_1}}^*}/{A_{\pi'_{\ell_1}}^{*3}}\big)_{N=1}$\end{small} gives  \begin{small}$(\overline{{\ell_2}^2\pi_{\ell_1}^2},\overline{{\ell_2}\pi_{\ell_1}}) \notin \text{Im } \delta_{\phi,K_{\pi'_{\ell_1}}}$.\end{small} A similar argument shows that  \begin{small}$ (\overline{\ell_2^2\pi_{\ell_1}},\overline{\ell_2\pi_{\ell_1}^2}) \notin \text{Im } \delta_{\phi,K_{\pi'_{\ell_1}}}.$\end{small} Thus one can deduce that \begin{small}$\dim_{\F_3} {\rm Sel}^\phi(E_{16n^2}/K) =2$,\end{small} when \begin{small}$\ell_1 \equiv 1 \pmod 9$,\end{small} \begin{small}$\ell_2 \equiv 8 \pmod 9$\end{small} and \begin{small}$\big(\frac{\ell_2}{\pi_{\ell_1}}\big)_3 \neq 1 $.\end{small}\\
     
     \item Let \begin{small}$\ell_1 \equiv 4 \pmod 9$\end{small} and \begin{small}$\ell_2 \equiv 2 \pmod 9$.\end{small} Proving that \begin{small}$\dim_{\F_3} {\rm Sel}^\phi(E_{16n^2}/K) =2$\end{small} involves checking of the $13$ subgroups mentioned in the proof of (2) in Prop. \ref{sylvesterpq2} for their containment in \begin{small}${\rm Sel}^\phi(E_{16n^2}/K)$.\end{small} We observe the following:
     \begin{enumerate}
         \item If \begin{small}$\big(\frac{\ell_2}{\pi_{\ell_1}}\big)_3 = 1, $\end{small} then \begin{small}$ (\overline{\ell_2^2\pi_{\ell_1}},\overline{\ell_2\pi_{\ell_1}^2}) \in {\rm Sel}^\phi(E_{16n^2}/K)$.\end{small} The elements  \begin{small}$(\overline{\zeta^2\ell_2^2\pi_{\ell_1}^2},\overline{\zeta\ell_2\pi_{\ell_1}}),$\end{small}  \begin{small}$ (\overline{\zeta\ell_2^2\pi_{\ell_1}^2},\overline{\zeta^2\ell_2\pi_{\ell_1}})$, $ (\overline{\zeta^2\ell_2\pi_{\ell_1}^2},\overline{\zeta\ell_2^2\pi_{\ell_1}})$\end{small}    \begin{small}$(\overline{\zeta^2\ell_2^2\pi_{\ell_1}},\overline{\zeta\ell_2\pi_{\ell_1}^2})$\end{small} and \begin{small}$(\overline{\zeta^2},\overline{\zeta})$\end{small} are not in \begin{small}$\text{Im } \delta_{\phi,K_{\pi'_{\ell_1}}}$,\end{small} whereas all the other elements are not in \begin{small}$\text{Im } \delta_{\phi,K_{\p}}$.\end{small} So, \begin{small}$\dim_{\F_3} {\rm Sel}^\phi(E_{16n^2}/K)=2$.\end{small} 
         
         \item If \begin{small}$\big(\frac{\ell_2}{\pi_{\ell_1}}\big)_3 =\zeta, $\end{small} then \begin{small}$ (\overline{\zeta^2\ell_2^2\pi_{\ell_1}},\overline{\zeta\ell_2\pi_{\ell_1}^2}) \in {\rm Sel}^\phi(E_{16n^2}/K)$.\end{small} Note that \begin{small}$(\overline{\pi_{\ell_1}}^2,\overline{\pi_{\ell_1}}) \notin \text{Im } \delta_{\phi,K_{{\ell_2}}}$\end{small} and \begin{small}$(\overline{\zeta^2\ell_2^2\pi_{\ell_1}^2},\overline{\zeta\ell_2\pi_{\ell_1}})$, $(\overline{\zeta^2\ell_2^2},\overline{\zeta\ell_2}) \notin \text{Im } \delta_{\phi,K_{\pi_{\ell_1}}}$,\end{small} while all the other elements are not in \begin{small}$\text{Im } \delta_{\phi,K_{\pi'_{\ell_1}}}$.\end{small} Thus, \begin{small}$ \dim_{\F_3} {\rm Sel}^\phi(E_{16n^2}/K)=2$.\end{small} 
         
         \item If \begin{small}$\big(\frac{\ell_2}{\pi_{\ell_1}}\big)_3 =\zeta^2, $\end{small} then \begin{small}$ (\overline{\zeta^2\ell_2\pi_{\ell_1}^2},\overline{\zeta\ell_2^2\pi_{\ell_1}}) \in {\rm Sel}^\phi(E_{16n^2}/K)$.\end{small} Notice that \begin{small}$(\overline{\pi_{\ell_1}}^2,\overline{\pi_{\ell_1}}) \notin \text{Im } \delta_{\phi,K_{{\ell_2}}}$\end{small} and \begin{small}$(\overline{\zeta\ell_2^2\pi_{\ell_1}^2},\overline{\zeta^2\ell_2\pi_{\ell_1}})$, $(\overline{\zeta\ell_2^2},\overline{\zeta^2\ell_2}) \notin \text{Im } \delta_{\phi,K_{\pi_{\ell_1}}}$,\end{small} whereas all the remaining elements are not in \begin{small}$\text{Im } \delta_{\phi,K_{\pi'_{\ell_1}}}$.\end{small} Thus, \begin{small}$ \dim_{\F_3} {\rm Sel}^\phi(E_{16n^2}/K)=2$.\end{small} \\
     \end{enumerate}
     
     \item Let \begin{small}$\ell_1 \equiv 7 \pmod 9$\end{small} and \begin{small}$\ell_2 \equiv 5 \pmod 9$.\end{small} Checking the $13$ subgroups mentioned in the proof of (2) in Prop. \ref{sylvesterpq2} for their containment in \begin{small}${\rm Sel}^\phi(E_{16n^2}/K)$\end{small} yields: 
     \begin{enumerate}
         \item If \begin{small}$\big(\frac{\ell_2}{\pi_{\ell_1}}\big)_3 = 1, $\end{small} then \begin{small}$ (\overline{\ell_2^2\pi_{\ell_1}},\overline{\ell_2\pi_{\ell_1}^2}) \in {\rm Sel}^\phi(E_{16n^2}/K)$.\end{small} Note that \begin{small}$(\overline{\zeta^2\ell_2^2},\overline{\zeta\ell_2})$,$(\overline{\zeta\ell_2^2},\overline{\zeta^2\ell_2}) \notin \text{Im } \delta_{\phi,K_{\pi_{\ell_1}}}$\end{small} and \begin{small}$(\overline{\ell_2^2},\overline{\ell_2})$, $(\overline{\pi_{\ell_1}^2},\overline{\pi_{\ell_1}})$, $(\overline{\ell_2^2\pi_{\ell_1}^2},\overline{\ell_2\pi_{\ell_1}}) \notin \text{Im } \delta_{\phi,K_\p}$,\end{small} whereas all the remaining elements are not in \begin{small}$\text{Im } \delta_{\phi,K_{\pi'_{\ell_1}}}$.\end{small} Thus, \begin{small}$ \dim_{\F_3} {\rm Sel}^\phi(E_{16n^2}/K)=2$.\end{small} 
         
         \item If \begin{small}$\big(\frac{\ell_2}{\pi_{\ell_1}}\big)_3 =\zeta, $\end{small} then \begin{small}$ (\overline{\zeta^2\ell_2\pi_{\ell_1}^2},\overline{\zeta\ell_2^2\pi_{\ell_1}}) \in {\rm Sel}^\phi(E_{16n^2}/K)$.\end{small} Observe that \begin{small}$(\overline{\pi_{\ell_1}^2},\overline{\pi_{\ell_1}}) \notin \text{Im } \delta_{\phi,K_{{\ell_2}}}$,\end{small}  \begin{small}$(\overline{\zeta\ell_2^2\pi_{\ell_1}^2},\overline{\zeta^2\ell_2\pi_{\ell_1}}) \notin \text{Im } \delta_{\phi,K_\p}$\end{small} and \begin{small}$
        (\overline{\zeta\ell_2^2},\overline{\zeta^2\ell_2}) \notin \text{Im } \delta_{\phi,K_{\pi_{\ell_1}}}$,\end{small} whereas all the other elements are not in \begin{small}$\text{Im } \delta_{\phi,K_{\pi'_{\ell_1}}}$.\end{small} Thus, \begin{small}$ \dim_{\F_3} {\rm Sel}^\phi(E_{16n^2}/K)=2$.\end{small} 
         
         \item If \begin{small}$\big(\frac{\ell_2}{\pi_{\ell_1}}\big)_3 =\zeta^2, $\end{small} then \begin{small}$ (\overline{\zeta^2\ell_2^2\pi_{\ell_1}},\overline{\zeta\ell_2\pi_{\ell_1}^2}) \in {\rm Sel}^\phi(E_{16n^2}/K)$.\end{small} Notice that \begin{small}$(\overline{\pi_{\ell_1}^2},\overline{\pi_{\ell_1}}) \notin \text{Im } \delta_{\phi,K_{{\ell_2}}}$,\end{small}  \begin{small}$(\overline{\zeta^2\ell_2^2\pi_{\ell_1}^2},\overline{\zeta\ell_2\pi_{\ell_1}}) \notin \text{Im } \delta_{\phi,K_\p}$\end{small} and \begin{small}$
        (\overline{\zeta^2\ell_2^2},\overline{\zeta\ell_2}) \notin \text{Im } \delta_{\phi,K_{\pi_{\ell_1}}}$,\end{small} while all the remaining elements are not in \begin{small}$\text{Im } \delta_{\phi,K_{\pi'_{\ell_1}}}$.\end{small} Thus, \begin{small}$ \dim_{\F_3} {\rm Sel}^\phi(E_{16n^2}/K)=2$.\end{small} 
     \end{enumerate}
\end{enumerate}
This shows that \begin{small}$ \dim_{\F_3} {\rm Sel}^\phi(E_{16n^2}/K)=2$\end{small} in all the cases, except when \begin{small}$\ell_1 \equiv 1 \pmod 9$, $\ell_2 \equiv 8 \pmod 9 \text{ and } \big(\frac{\ell_2}{\pi_{\ell_1}}\big)_3 = 1$;\end{small} in which case it is $4$. This completes the proof for \begin{small}$n=\ell_1\ell_2$\end{small}. 
\end{proof}

\subsection{The Case $k_1=2$, that is, $\ell_1 \equiv \ell_2 \equiv 1 \pmod 3$}
We are now only left with the case when both \begin{small}$\ell_1$\end{small} and \begin{small}$\ell_2$\end{small} are primes congruent to \begin{small}$1 \pmod 3$.\end{small} Let \begin{small}$\ell_i$\end{small} split as \begin{small}$\ell_i=\pi_{\ell_i}\pi'_{\ell_i}$\end{small} for \begin{small}$i=1, 2$\end{small} in \begin{small}$K$\end{small} and denote by \begin{small}$\big(\frac{\cdot}{\cdot}\big)_3$\end{small} the cubic residue symbol.

\begin{proposition}\label{sylvesterpq4}
 Let \begin{small}$n$\end{small} be an integer of the form  \begin{small}$\ell_1\ell_2$\end{small} or \begin{small}$\ell_1^2\ell_2$\end{small} or \begin{small}$\ell_1^2\ell_2^2$,\end{small} where \begin{small}$\ell_1 \equiv \ell_2 \equiv 1 \pmod 3$\end{small} are distinct primes and \begin{small}$\phi: E_{16n^2} \to E_{16n^2}$\end{small} be the $3$-isogeny  defined in \eqref{eq:defofphi}. Then
 \begin{enumerate}
     \item If \begin{small}$n \equiv  1 \pmod 9$,\end{small} then 
     \begin{small}$$\dim_{\F_3} {\rm Sel}^\phi(E_{16n^2}/K) =\begin{cases}
     5, & \text{if } \ell_1 \equiv \ell_2 \equiv 1 \pmod 9,\text{  }  \big(\frac{\pi_{\ell_1}}{\pi_{\ell_2}}\big)_3 = 1 \text{ and } \big(\frac{\pi'_{\ell_1}}{\pi_{\ell_2}}\big)_3 = 1, \\
     3, & \text{otherwise.}
     \end{cases}$$ \end{small}
     \item If \begin{small}$n \equiv 4 \text{ or } 7 \pmod 9$,\end{small} then 
     \begin{small}$$\dim_{\F_3} {\rm Sel}^\phi(E_{16n^2}/K) =\begin{cases}
     4, & \text{if }  \big(\frac{\pi_{\ell_1}}{\pi_{\ell_2}}\big)_3 = 1 \text{ and } \big(\frac{\pi'_{\ell_1}}{\pi_{\ell_2}}\big)_3 = 1, \\
     2, & \text{otherwise.}
     \end{cases}$$ \end{small}
 \end{enumerate}
\end{proposition}

\proof We prove the result for \begin{small}$n=\ell_1\ell_2$.\end{small}
    Here \begin{small}$S_{n}=\{ \pi_{\ell_1}, \pi'_{\ell_1}, \pi_{\ell_2}, \pi'_{\ell_2} \}$\end{small} and the group  \begin{small}$\OO_{K,S_n}^* = \langle \pm\zeta, \pi_{\ell_1}, \pi'_{\ell_1}, \pi_{\ell_2}, \pi'_{\ell_2} \rangle = \langle \pm\zeta, \pi_{\ell_1}, \pi'_{\ell_1}, \pi_{\ell_2}, n \rangle$.\end{small} Hence, we have \begin{small} ${\rm Sel}^\phi(E_{16n^2}/K) \subset \Big(\frac{A_{S_{n}}^*}{A_{S_{n}}^{*3}}\Big)_{N=1} = \big\langle (\overline{\zeta^2},\overline{\zeta}), (\overline{\pi_{\ell_1}}^2,\overline{\pi_{\ell_1}}), (\overline{\pi'^{2}_{\ell_1}},\overline{\pi'_{\ell_1}}), (\overline{\pi_{\ell_2}}^2,\overline{\pi_{\ell_2}}),(\overline{n}^2, \overline{n}) \big\rangle $ \end{small} which implies by Lemma \ref{selmerbound} that \begin{small}$\dim_{\F_3} {\rm Sel}^\phi(E_{16n^2}/K) \le |S_n|+1= 5$.\end{small} 

    As before, we observe that the element \begin{small}$(\overline{n}^2, \overline{n}) \in {\rm Sel}^\phi(E_{16n^2}/K)$\end{small} in every case. In the cases where we claim \begin{small}$\dim_{\F_3}{\rm Sel}^\phi(E_{16n^2}/K)$\end{small} is $2$ or $3$, in order to determine  \begin{small}$\dim_{\F_3}{\rm Sel}^\phi(E_{16n^2}/K)$\end{small} precisely, one has to investigate the $40$ subgroups of order $3$ of  \begin{small}$\Big(\frac{A_{S_{n}}^*}{A_{S_{n}}^{*3}}\Big)_{N=1} \Big/ \big\langle (\overline{n}^2, \overline{n}) \big\rangle$\end{small} corresponding to the following elements: \begin{small}$\zeta$, $\pi_{\ell_1}$, $\pi'_{\ell_1}$, $\pi_{\ell_2}$, $\ell_1$, $\pi_{\ell_1}\pi_{\ell_2}$, $\pi'_{\ell_1}\pi_{\ell_2}$, $\ell_1\pi_{\ell_1}$, $\pi_{\ell_1}^2\pi_{\ell_2}$, $\pi'^{2}_{\ell_1}\pi_{\ell_2}$, $\ell_1\pi_{\ell_2}$, $\ell_1^2\pi_{\ell_2}$, $\ell_1\pi_{\ell_1}\pi_{\ell_2}$, $\ell_1\pi'_{\ell_1}\pi_{\ell_2}$, $\zeta\pi_{\ell_1}$, $\zeta\pi'_{\ell_1}$, $\zeta\pi_{\ell_2}$, $\zeta\ell_1$, $\zeta\pi_{\ell_1}\pi_{\ell_2}$, $\zeta\pi'_{\ell_1}\pi_{\ell_2}$, $\zeta\ell_1\pi_{\ell_1}$, $\zeta\pi_{\ell_1}^2\pi_{\ell_2}$, $\zeta\pi'^{2}_{\ell_1}\pi_{\ell_2}$, $\zeta\ell_1\pi_{\ell_2}$, $\zeta\ell_1^2\pi_{\ell_2}$, $\zeta\ell_1\pi_{\ell_1}\pi_{\ell_2}$, $\zeta\ell_1\pi'_{\ell_1}\pi_{\ell_2}$, $\zeta^2\pi_{\ell_1}$, $\zeta^2\pi'_{\ell_1}$, $\zeta^2\pi_{\ell_2}$, $\zeta^2\ell_1$, $\zeta^2\pi_{\ell_1}\pi_{\ell_2}$, $\zeta^2\pi'_{\ell_1}\pi_{\ell_2}$, $\zeta^2\ell_1\pi_{\ell_1}$, $\zeta^2\pi_{\ell_1}^2\pi_{\ell_2}$, $\zeta^2\pi'^{2}_{\ell_1}\pi_{\ell_2}$, $\zeta^2\ell_1\pi_{\ell_2}$, $\zeta^2\ell_1^2\pi_{\ell_2}$, $\zeta^2\ell_1\pi_{\ell_1}\pi_{\ell_2}$, $\zeta^2\ell_1\pi'_{\ell_1}\pi_{\ell_2}$. \end{small} \\
    
    A similar but much tedious computation as in Props. \ref{sylvesterpq}-\ref{sylvesterpq3} shows that
    \begin{enumerate}
        \item All of \begin{small}$(\overline{\zeta^2},\overline{\zeta})$, $ (\overline{\pi_{\ell_1}}^2,\overline{\pi_{\ell_1}})$, $ (\overline{\pi'^{2}_{\ell_1}},\overline{\pi'_{\ell_1}})$, $ (\overline{\pi_{\ell_2}}^2,\overline{\pi_{\ell_2}})$\end{small} belong to \begin{small}${\rm Sel}^\phi(E_{16n^2}/K)$,\end{small} when  \begin{small}$\ell_1 \equiv \ell_2 \equiv 1 \pmod 9$\end{small} and both \begin{small}$ \big(\frac{\pi_{\ell_1}}{\pi_{\ell_2}}\big)_3$ \end{small} and \begin{small}$\big(\frac{\pi'_{\ell_1}}{\pi_{\ell_2}}\big)_3 $\end{small} are $1$. Hence, \begin{small}$\dim_{\F_3} {\rm Sel}^\phi(E_{16n^2}/K) = 5$\end{small} here.\\
        
       \item Let \begin{small}$\ell_1 \equiv \ell_2 \equiv 1 \pmod 9$\end{small} such that at least one of \begin{small}$ \big(\frac{\pi_{\ell_1}}{\pi_{\ell_2}}\big)_3$\end{small} and \begin{small}$\big(\frac{\pi'_{\ell_1}}{\pi_{\ell_2}}\big)_3$\end{small} is not $1$. We see that apart from \begin{small}$(\overline{n}^2, \overline{n})$\end{small}, out of the $40$ subgroups stated above, only the subgroups generated by elements mentioned below are contained in the \begin{small}$\phi$-\end{small}Selmer group:
        \begin{enumerate}
            \item If \begin{small}$ \big(\frac{\pi_{\ell_1}}{\pi_{\ell_2}}\big)_3 = 1$ \end{small} and \begin{small}$\big(\frac{\pi'_{\ell_1}}{\pi_{\ell_2}}\big)_3 \neq 1$,\end{small} then  \begin{small}$(\overline{\zeta^2},\overline{\zeta})$\end{small},  \begin{small}$(\overline{\pi'^{2}_{\ell_1}\pi_{\ell_2}^2},\overline{\pi'_{\ell_1}\pi_{\ell_2}}) \in {\rm Sel}^\phi(E_{16n^2}/K)$.\end{small} 
            
            \item If \begin{small}$ \big(\frac{\pi_{\ell_1}}{\pi_{\ell_2}}\big)_3 = \big(\frac{\pi'_{\ell_1}}{\pi_{\ell_2}}\big)_3 = \zeta \text{ or } \zeta^2$,\end{small} then  \begin{small}$(\overline{\zeta^2},\overline{\zeta})$\end{small},  \begin{small}$(\overline{{\ell_1}\pi'_{\ell_1}},\overline{{\ell_1}\pi_{\ell_1}}) \in {\rm Sel}^\phi(E_{16n^2}/K)$.\end{small} 

            \item If \begin{small}$ \big(\frac{\pi_{\ell_1}}{\pi_{\ell_2}}\big)_3 = \zeta$\end{small} and  \begin{small}$\big(\frac{\pi'_{\ell_1}}{\pi_{\ell_2}}\big)_3 = \zeta^2$,\end{small} then  \begin{small}$(\overline{\zeta^2},\overline{\zeta})$\end{small},  \begin{small}$(\overline{{\ell_1}^2\pi_{\ell_2}},\overline{{\ell_1}\pi_{\ell_2}^2}) \in {\rm Sel}^\phi(E_{16n^2}/K)$.\end{small}
        \end{enumerate}
        \noindent Thus, \begin{small}$\dim_{\F_3} {\rm Sel}^\phi(E_{16n^2}/K) = 3$\end{small} in these cases.\\
        
       \item Next, let \begin{small}$\ell_1 \equiv 4 \pmod 9$\end{small} and  \begin{small}$\ell_2 \equiv 7 \pmod 9$.\end{small} Then the elements given below, along with \begin{small}$(\overline{n}^2, \overline{n})$\end{small}, generate \begin{small}${\rm Sel}^\phi(E_{16n^2}/K)$:\end{small}
        \begin{enumerate}
            \item If both \begin{small}$ \big(\frac{\pi_{\ell_1}}{\pi_{\ell_2}}\big)_3 = \big(\frac{\pi'_{\ell_1}}{\pi_{\ell_2}}\big)_3 = 1$,\end{small} then  \begin{small}$(\overline{\ell_1 \pi'_{\ell_1}},\overline{\ell_1\pi_{\ell_1}})$\end{small},  \begin{small}$(\overline{\pi'^{2}_{\ell_1}\pi_{\ell_2}^2},\overline{\pi'_{\ell_1}\pi_{\ell_2}}) \in {\rm Sel}^\phi(E_{16n^2}/K)$.\end{small} 
            
            \item If \begin{small}$\big(\frac{\pi_{\ell_1}}{\pi_{\ell_2}}\big)_3 = 1$\end{small} and \begin{small}$\big(\frac{\pi'_{\ell_1}}{\pi_{\ell_2}}\big)_3 = \zeta$,\end{small} then  \begin{small}$(\overline{\pi'^{2}_{\ell_1}\pi_{\ell_2}^2},\overline{\pi'_{\ell_1}\pi_{\ell_2}})$\end{small},  \begin{small}$(\overline{\zeta^2{\ell_1}\pi'_{\ell_1}},\overline{\zeta{\ell_1}\pi_{\ell_1}}) \in {\rm Sel}^\phi(E_{16n^2}/K)$.\end{small} 

            \item If \begin{small}$ \big(\frac{\pi_{\ell_1}}{\pi_{\ell_2}}\big)_3 = 1$\end{small} and \begin{small}$\big(\frac{\pi'_{\ell_1}}{\pi_{\ell_2}}\big)_3 = \zeta^2$,\end{small} then  \begin{small}$(\overline{\pi'^{2}_{\ell_1}\pi_{\ell_2}^2},\overline{\pi'_{\ell_1}\pi_{\ell_2}})$\end{small},  \begin{small}$(\overline{\zeta^2{\ell_1}\pi_{\ell_1}},\overline{\zeta{\ell_1}\pi'_{\ell_1}}) \in {\rm Sel}^\phi(E_{16n^2}/K)$.\end{small} 

            \item If both \begin{small}$ \big(\frac{\pi_{\ell_1}}{\pi_{\ell_2}}\big)_3 = \big(\frac{\pi'_{\ell_1}}{\pi_{\ell_2}}\big)_3 = \zeta$,\end{small} then    \begin{small}$(\overline{{\ell_1}\pi'_{\ell_1}},\overline{{\ell_1}\pi_{\ell_1}})$,\end{small} \begin{small}$(\overline{\zeta^2\pi'_{\ell_1}\pi_{\ell_2}},\overline{\zeta\pi'^{2}_{\ell_1}\pi_{\ell_2}^2}) \in {\rm Sel}^\phi(E_{16n^2}/K)$.\end{small}

            \item If \begin{small}$ \big(\frac{\pi_{\ell_1}}{\pi_{\ell_2}}\big)_3 = \zeta$\end{small} and  \begin{small}$\big(\frac{\pi'_{\ell_1}}{\pi_{\ell_2}}\big)_3 = \zeta^2$,\end{small} then    \begin{small}$(\overline{{\ell_1}^2\pi_{\ell_2}},\overline{{\ell_1}\pi_{\ell_2}^2}),$\end{small} \begin{small}$(\overline{\zeta^2{\ell_1}\pi'_{\ell_1}},\overline{\zeta{\ell_1}\pi_{\ell_1}}) \in {\rm Sel}^\phi(E_{16n^2}/K)$.\end{small}

            \item If both \begin{small}$ \big(\frac{\pi_{\ell_1}}{\pi_{\ell_2}}\big)_3 = \big(\frac{\pi'_{\ell_1}}{\pi_{\ell_2}}\big)_3 = \zeta^2$,\end{small} then  \begin{small}$(\overline{{\ell_1}\pi'_{\ell_1}},\overline{{\ell_1}\pi_{\ell_1}})$,\end{small} \begin{small}$(\overline{\zeta^2\pi'^{2}_{\ell_1}\pi_{\ell_2}^2},\overline{\zeta\pi'_{\ell_1}\pi_{\ell_2}}) \in {\rm Sel}^\phi(E_{16n^2}/K)$.\end{small} 
        \end{enumerate}
        \noindent Hence, \begin{small}$\dim_{\F_3} {\rm Sel}^\phi(E_{16n^2}/K) = 3$\end{small} in all of these cases.\\

         \item When \begin{small}$n=\ell_1\ell_2 \not\equiv 1 \pmod 9,$\end{small} then at least one of \begin{small}$\ell_1$\end{small} and \begin{small}$\ell_2$\end{small} is not \begin{small}$1 \pmod 9.$\end{small} We know that \begin{small}$\zeta_9 \notin K_{\pi_{\ell_i}}^*$\end{small} whenever \begin{small}$\ell_i \not\equiv \pm 1 \pmod 9.$\end{small} Hence, we get that \begin{small}$(\overline{\zeta^2},\overline{\zeta}) \notin \text{Im } \delta_{\phi,K_{\pi_{\ell_i}}}$\end{small} and so \begin{small}$(\overline{\zeta^2},\overline{\zeta}) \notin {\rm Sel}^\phi(E_{16n^2}/K)$.\end{small} We also obtain that all of \begin{small}$(\overline{\pi_{\ell_1}}^2,\overline{\pi_{\ell_1}})$, $ (\overline{\pi'^{2}_{\ell_1}},\overline{\pi'_{\ell_1}})$, $ (\overline{\pi_{\ell_2}}^2,\overline{\pi_{\ell_2}})$\end{small} belong to \begin{small}${\rm Sel}^\phi(E_{16n^2}/K)$,\end{small} if \begin{small}$n \not\equiv 1 \pmod 9$\end{small} and each of \begin{small}$ \big(\frac{\pi_{\ell_1}}{\pi_{\ell_2}}\big)_3$ \end{small} and \begin{small}$\big(\frac{\pi'_{\ell_1}}{\pi_{\ell_2}}\big)_3 $\end{small} is $1$. Hence, \begin{small}$\dim_{\F_3} {\rm Sel}^\phi(E_{16n^2}/K) = 4$\end{small} in this case.\\

        \item Finally, when \begin{small}$n \not\equiv 1 \pmod 9$\end{small} and at least one of \begin{small}$ \big(\frac{\pi_{\ell_1}}{\pi_{\ell_2}}\big)_3$\end{small} and \begin{small}$\big(\frac{\pi'_{\ell_1}}{\pi_{\ell_2}}\big)_3$\end{small} is not $1$, we obtain that \begin{small}${\rm Sel}^\phi(E_{16n^2}/K)$\end{small} is generated by \begin{small}$(\overline{n}^2, \overline{n})$\end{small} and the elements given below:
        \begin{enumerate}
            \item If \begin{small}$ \big(\frac{\pi_{\ell_1}}{\pi_{\ell_2}}\big)_3 = 1$ \end{small} and \begin{small}$\big(\frac{\pi'_{\ell_1}}{\pi_{\ell_2}}\big)_3 \neq 1$,\end{small} then  \begin{small}$(\overline{\pi'^{2}_{\ell_1}\pi_{\ell_2}^2},\overline{\pi'_{\ell_1}\pi_{\ell_2}}) \in {\rm Sel}^\phi(E_{16n^2}/K)$.\end{small} 
            
            \item If \begin{small}$ \big(\frac{\pi_{\ell_1}}{\pi_{\ell_2}}\big)_3 = \big(\frac{\pi'_{\ell_1}}{\pi_{\ell_2}}\big)_3 = \zeta \text{ or } \zeta^2$,\end{small} then \begin{small}$(\overline{{\ell_1}\pi'_{\ell_1}},\overline{{\ell_1}\pi_{\ell_1}}) \in {\rm Sel}^\phi(E_{16n^2}/K)$.\end{small} 

            \item If \begin{small}$ \big(\frac{\pi_{\ell_1}}{\pi_{\ell_2}}\big)_3 = \zeta$\end{small} and  \begin{small}$\big(\frac{\pi'_{\ell_1}}{\pi_{\ell_2}}\big)_3 = \zeta^2$,\end{small} then  \begin{small}$(\overline{{\ell_1}^2\pi_{\ell_2}},\overline{{\ell_1}\pi_{\ell_2}^2}) \in {\rm Sel}^\phi(E_{16n^2}/K)$.\end{small}
        \end{enumerate}
        \noindent Thus, \begin{small}$\dim_{\F_3} {\rm Sel}^\phi(E_{16n^2}/K) = 2$\end{small} in all of these cases.
        \qed
    \end{enumerate}

\section{Computation of the rank of the elliptic curve $E_{16n^2}/\Q$}\label{rankcomp}
We are now ready to obtain results on the Mordell-Weil rank of the curves \begin{small}$E_{-432n^2}$\end{small} over \begin{small}$\Q$.\end{small} 
We use Props. \ref{sylvesterpq}-\ref{sylvesterpq4} and Lemma \ref{rankbound} to obtain the results stated in Conjecture \ref{conjforrank2}. Recall that \begin{small}$\text{rk } E_{16n^2}(\Q)=\text{rk } E_{-432n^2}(\Q)$.\end{small}

\begin{theorem}\label{thmtosylvester} 
    Let \begin{small}$n$\end{small} be a cube-free integer co-prime to $3$ having  exactly two distinct prime factors \begin{small}$\ell_1$\end{small} and \begin{small}$\ell_2$.\end{small} We have \begin{small}$\text{rk } E_{-432n^2}(\Q)=0$\end{small} i.e. \begin{small}$n$\end{small} is not a rational cube sum in the following cases: 
    \begin{enumerate}
        \item \begin{small}$n=\ell_1\ell_2$\end{small} or \begin{small}$n=\ell_1^2\ell_2^2$\end{small} with \begin{small}$\ell_1 \equiv 2 \pmod 9$\end{small} and \begin{small}$\ell_2 \equiv 5 \pmod 9$,\end{small}

        \item \begin{small}$n=\ell_1^2\ell_2$\end{small} with \begin{small}$\ell_1 \equiv \ell_2 \equiv 2,5 \pmod 9$\end{small},
        
        \item \begin{small}$n=\ell_1\ell_2$\end{small} or \begin{small}$n=\ell_1^2\ell_2$\end{small} with  \begin{small}$n \equiv 2,5 \pmod 9$\end{small} where \begin{small}$\ell_1 \equiv 1 \pmod 3,$ $ \ell_2 \equiv 2 \pmod 3$\end{small} and \begin{small}$\ell_2$\end{small} is not a cube in \begin{small}$\F_{\ell_1}$\end{small},

        \item \begin{small}$n=\ell_1^2\ell_2^2$\end{small} or \begin{small}$n=\ell_1\ell_2^2$\end{small} with  \begin{small}$n \equiv 4,7 \pmod 9$\end{small} where \begin{small}$\ell_1 \equiv 1 \pmod 3,$ $ \ell_2 \equiv 2 \pmod 3$\end{small} and \begin{small}$\ell_2$\end{small} is not a cube in \begin{small}$\F_{\ell_1}$\end{small}. 
       \end{enumerate} 
\end{theorem}
\begin{proof}
From Props. \ref{sylvesterpq}-\ref{sylvesterpq3}, we have \begin{small}$\dim_{\F_3} {\rm Sel}^\phi(E_{16n^2}/K) =1$\end{small} if \begin{small}$n$\end{small} satisfies any of the conditions (1) to (4) above. Then Lemma \ref{rankbound} implies \begin{small}$\text{rk } E_{16n^2}(\Q)=\text{rk } E_{-432n^2}(\Q)=0$.\end{small}
\end{proof}

Here on-wards, we assume that \begin{small}$\dim_{\F_3} \Sh(E_{16n^2}/\Q)[3]$\end{small} is even. The following results are then immediate consequences of Props. \ref{sylvesterpq}-\ref{sylvesterpq4} and Lemma \ref{rankbound}:
\begin{theorem}\label{thmtosylvesterrk1}
   Let \begin{small}$n$\end{small} be a cube-free integer co-prime to $3$ having  exactly two distinct prime factors \begin{small}$\ell_1$\end{small} and \begin{small}$\ell_2$.\end{small} 
   Then we have \begin{small}$\text{rk } E_{-432n^2}(\Q) \le 1$\end{small} in  the following cases. 
  \begin{enumerate}
        \item \begin{small}$n \not\equiv \pm 1 \pmod 9$, $\ell_1 \equiv \ell_2 \equiv 2 \pmod 3$,\end{small}
        \item \begin{small}$n \equiv \pm 1 \pmod 9$, $\ell_1 \equiv 1 \pmod 9$, $\ell_2 \equiv 8 \pmod 9$\end{small} but \begin{small}$\ell_2$\end{small} is not a cube in \begin{small}$\F_{\ell_1}$,\end{small}
        \item \begin{small}$n \equiv \pm 1 \pmod 9$, $\ell_2 \equiv 2 \pmod 3$ , $\ell_1 \equiv 4 \text{ or } 7 \pmod 9$,\end{small}
        \item \begin{small}$n \not\equiv \pm 1 \pmod 9$, $\ell_1 \equiv \ell_2 \equiv 1 \pmod 3$\end{small} and at least one among \begin{small}$\pi_{\ell_1}, \pi_{\ell_1}'$\end{small} is not a cube in \begin{small}$\F_{\ell_2}$.\end{small}
\end{enumerate}
Further, if we assume that \begin{small}$\dim_{\F_3} \Sh(E_{16n^2}/\Q)[3]$\end{small} is even in the cases above, then \begin{small}$\text{rk } E_{-432n^2}(\Q)=1$\end{small} i.e. \begin{small}$n$\end{small} is a sum of two rational cubes.  
\end{theorem}

\begin{proof}
We know that \begin{small}$\dim_{\F_3} {\rm Sel}^\phi(E_{16n^2}/K) =2$\end{small} if \begin{small}$n$\end{small} satisfies any one of the conditions (1) to (4) stated above, from Props. \ref{sylvesterpq}-\ref{sylvesterpq4}.  The statement follows from Lemma \ref{rankbound}. 
\end{proof}

Theorems \ref{thmtosylvester} and \ref{thmtosylvesterrk1} can prove to be more efficient than SageMath \cite{sg}, as in many of the situations mentioned above \cite{sg} can only provide bounds on the rank of the elliptic curve \begin{small}$E_{-432n^2}$,\end{small} whereas these theorems allow us to find \begin{small}$\text{rk } E_{-432n^2}(\Q)$\end{small} precisely. We establish this claim in the examples below: 

\noindent {\bf Example (1)} Let \begin{small}$\ell_1=2$,\end{small} \begin{small}$\ell_2=131$\end{small} and  \begin{small}$n=\ell_1\ell_2$\end{small} or \begin{small}$n=\ell_1^2\ell_2^2$\end{small}. Then by Theorem \ref{thmtosylvester}(1), we get that \begin{small}$\text{rk } E_{-432n^2}(\Q)=0$.\end{small} Whereas using \cite{sg}, one can only obtain \begin{small}$0 \le \text{rk } E_{-432n^2}(\Q) \le 2$\end{small}. 

\noindent {\bf Example (2)} Let \begin{small}$\ell_1=11$\end{small} and \begin{small}$\ell_2=29$.\end{small} Then for \begin{small}$n=\ell_1^2\ell_2$\end{small}, by Theorem \ref{thmtosylvester}(2), we get that \begin{small}$\text{rk } E_{-432n^2}(\Q)=0$.\end{small} Whereas, \cite{sg} only predicts that \begin{small}$0 \le \text{rk } E_{-432n^2}(\Q) \le 2$\end{small}. 

\noindent {\bf Example (3)} Let \begin{small}$\ell_1=5$\end{small} and \begin{small}$\ell_2=41$.\end{small} Then for \begin{small}$n=\ell_1^2\ell_2$\end{small}, by Theorem \ref{thmtosylvester}(2), we get that \begin{small}$\text{rk } E_{-432n^2}(\Q)=0$.\end{small} Whereas, \cite{sg} only predicts that \begin{small}$0 \le \text{rk } E_{-432n^2}(\Q) \le 2$\end{small}.

\noindent {\bf Example (4)} Let \begin{small}$\ell_1=19$\end{small} and \begin{small}$\ell_2=317$.\end{small} Then it is easy to see that \begin{small}$\ell_2$\end{small} is not a cube in \begin{small}$\F_{\ell_1}$.\end{small} So for \begin{small}$n=\ell_1\ell_2$\end{small} or \begin{small}$n=\ell_1^2\ell_2$\end{small}, by Theorem \ref{thmtosylvester}(3), we get that \begin{small}$\text{rk } E_{-432n^2}(\Q)=0$.\end{small} On the other hand, \cite{sg} only provides the bounds \begin{small}$0 \le \text{rk } E_{-432n^2}(\Q) \le 2$\end{small}. 

\noindent {\bf Example (5)} Let \begin{small}$\ell_1=37$\end{small} and \begin{small}$\ell_2=131$.\end{small} Then one can see that \begin{small}$\ell_2$\end{small} is not a cube in \begin{small}$\F_{\ell_1}$.\end{small} Hence for \begin{small}$n=\ell_1\ell_2$\end{small} or \begin{small}$n=\ell_1^2\ell_2$\end{small}, by Theorem \ref{thmtosylvester}(3), we get that \begin{small}$\text{rk } E_{-432n^2}(\Q)=0$.\end{small} On the other hand, \cite{sg} only provides the bounds \begin{small}$0 \le \text{rk } E_{-432n^2}(\Q) \le 2$\end{small}.

\noindent {\bf Example (6)} Let \begin{small}$\ell_1=97$\end{small} and \begin{small}$\ell_2=17$.\end{small} Then it can be easily seen that \begin{small}$\ell_2$\end{small} is not a cube in \begin{small}$\F_{\ell_1}$.\end{small} Hence for \begin{small}$n=\ell_1\ell_2^2$\end{small} or \begin{small}$n=\ell_1^2\ell_2^2$\end{small}, by Theorem \ref{thmtosylvester}(4), we get that \begin{small}$\text{rk } E_{-432n^2}(\Q)=0$.\end{small} On the other hand, \cite{sg} only provides the bounds \begin{small}$0 \le \text{rk } E_{-432n^2}(\Q) \le 2$\end{small}.

\noindent {\bf Example (7)} Let \begin{small}$\ell_1=281$\end{small} and \begin{small}$\ell_2=89$.\end{small} Then for \begin{small}$n \in \{\ell_1\ell_2, \ell_1^2\ell_2\}$,\end{small} we get \begin{small}$\text{rk } E_{-432n^2}(\Q) \le 1$\end{small} by Theorem \ref{thmtosylvesterrk1}(1). 
On the other hand, we obtain \begin{small}$1 \le \text{rk } E_{-432n^2}(\Q) \le 3$\end{small} using \cite{sg}, when \begin{small}$n \in \{\ell_1\ell_2, \ell_1^2\ell_2\}$.\end{small} Hence, we deduce that \begin{small}$\text{rk } E_{-432n^2}(\Q) = 1$.\end{small} 

\noindent {\bf Example (8)} Let \begin{small}$\ell_1=73$\end{small} and \begin{small}$\ell_2=269$.\end{small} We obtain \begin{small}$1 \le \text{rk } E_{-432\ell_1^4\ell_2^2}(\Q) \le 3$\end{small} using \cite{sg}. One can see that \begin{small}$\ell_2$\end{small} is not a cube in \begin{small}$\F_{\ell_1}$\end{small}, so by Theorem \ref{thmtosylvesterrk1}(2) we have \begin{small}$\text{rk } E_{-432\ell_1^4\ell_2^2}(\Q) \le 1$\end{small}  and hence we conclude that \begin{small}$\text{rk } E_{-432\ell_1^4\ell_2^2}(\Q) = 1$\end{small}.

\noindent {\bf Example (9)} Let \begin{small}$\ell_1=139$\end{small} and \begin{small}$\ell_2=389$.\end{small} Then for \begin{small}$n \in \{\ell_1\ell_2, \ell_1^2\ell_2^2\}$,\end{small} we get \begin{small}$\text{rk } E_{-432n^2}(\Q) \le 1$\end{small} by Theorem \ref{thmtosylvesterrk1}(3). 
On the other hand, we get \begin{small}$1 \le \text{rk } E_{-432n^2}(\Q) \le 3$\end{small} using \cite{sg}, when \begin{small}$n \in \{\ell_1\ell_2, \ell_1^2\ell_2^2\}$.\end{small} Thus, \begin{small}$\text{rk } E_{-432n^2}(\Q) = 1$.\end{small}

\noindent {\bf Example (10)} Let \begin{small}$\ell_1=157$\end{small} and \begin{small}$\ell_2=19$.\end{small} Let \begin{small}$n \in \{\ell_1^2\ell_2, \ell_1\ell_2^2\}$.\end{small}  One can see that \begin{small}$\ell_1$\end{small} is not a cube in \begin{small}$\F_{\ell_2}$\end{small}, which implies at least one of \begin{small}$\pi_{\ell_1}, \pi'_{\ell_1}$\end{small} is not a cube in \begin{small}$\F_{\ell_2}$\end{small}. Thus by Theorem \ref{thmtosylvesterrk1}(4), we have \begin{small}$\text{rk } E_{-432n^2}(\Q) \le 1$\end{small}. We obtain \begin{small}$1 \le \text{rk } E_{-432n^2}(\Q) \le 3$\end{small} using \cite{sg} and hence, we conclude that \begin{small}$\text{rk } E_{-432n^2}(\Q) = 1$\end{small}.

\begin{theorem}\label{thmtosylvesterrk3}
   Let \begin{small}$n$\end{small} be a cube-free integer co-prime to $3$ having  exactly two distinct prime factors \begin{small}$\ell_1$\end{small} and \begin{small}$\ell_2$.\end{small} In the following cases, we have \begin{small}$\text{rk } E_{-432n^2}(\Q) \le 3$.\end{small}
    \begin{enumerate}
        \item \begin{small}$n \equiv \pm 1 \pmod 9$, $\ell_1 \equiv 1 \pmod 9$, $\ell_2 \equiv 8 \pmod 3$\end{small} and \begin{small}$\ell_2$\end{small} is a cube in \begin{small}$\F_{\ell_1}$,\end{small}
        \item \begin{small}$n \not\equiv \pm 1 \pmod 9$, $\ell_1 \equiv \ell_2 \equiv 1 \pmod 3$\end{small} and both \begin{small}$\pi_{\ell_1}, \pi_{\ell_1}'$\end{small} are cubes in \begin{small}$\F_{\ell_2}$.\end{small}
\end{enumerate}
Further, if we assume that \begin{small}$\dim_{\F_3} \Sh(E_{16n^2}/\Q)[3]$\end{small} is even in the above cases, then we have \begin{small}$\text{rk } E_{-432n^2}(\Q)=1 \text{ or } 3$\end{small} i.e. \begin{small}$n$\end{small} is a rational cube sum.
\end{theorem}

\begin{proof}
If \begin{small}$n$\end{small} satisfies any one of the conditions (1) or (2) written above, then we know that \begin{small}$\dim_{\F_3} {\rm Sel}^\phi(E_{16n^2}/K) =4$\end{small} from Props. \ref{sylvesterpq3}-\ref{sylvesterpq4}.  The statement follows from Lemma \ref{rankbound}. 
\end{proof}

\noindent {\bf Example (1)} Let \begin{small}$\ell_1=19$\end{small} and \begin{small}$\ell_2=467$,\end{small} then we see that \begin{small}$\ell_2$\end{small} is a cube in \begin{small}$\F_{\ell_1}$\end{small}. For \begin{small}$n \in \{\ell_1\ell_2, \ell_1^2\ell_2, \ell_1\ell_2^2, \ell_1^2\ell_2^2\}$,\end{small} we have \begin{small}$n \equiv \pm 1 \pmod 9$.\end{small}
So, by Theorem \ref{thmtosylvesterrk3}(1), it then follows that  \begin{small}$\text{rk } E_{-432n^2}(\Q) \le 3$\end{small}. Further, if we assume that \begin{small}$\dim_{\F_3} \Sh(E_{16n^2}/\Q)[3]$\end{small} is even, then we have \begin{small}$\text{rk } E_{-432n^2}(\Q)=1 \text{ or } 3$\end{small}.

 Using \cite{sg}, we obtain  \begin{small}$\text{rk } E_{-432\ell_1^4\ell_2^2}(\Q)=3$ \end{small} and \begin{small}$\text{rk } E_{-432n^2}(\Q)=1$\end{small} for \begin{small}$n \in \{\ell_1\ell_2, \ell_1\ell_2^2, \ell_1^2\ell_2^2\}$.\end{small}

\noindent {\bf Example (2)} Let \begin{small}$\ell_1=103$\end{small} and \begin{small}$\ell_2=13$,\end{small} then  \begin{small}$\ell_1\ell_2 \equiv 7 \pmod 9$\end{small} and \begin{small}$\ell_1^2\ell_2^2 \equiv 4 \pmod 9$\end{small}. Notice that \begin{small}$\ell_1=103$\end{small} is a cube in \begin{small}$\F_{\ell_2}$\end{small}. Next observe that there exists an \begin{small}$\alpha \in \Z_{\ell_2}$\end{small} such that \begin{small}$\alpha \equiv 6 \pmod{13}$\end{small}. Writing \begin{small}$\pi_{\ell_1} = \frac{1}{2}(13+ 9 \alpha)$\end{small}, we see that \begin{small}$\pi_{\ell_1} \equiv 1 \pmod{13}$\end{small}. As a consequence, we see that both \begin{small}$\pi_{\ell_1}, \pi_{\ell_1}'$\end{small} are cubes in \begin{small}$\F_{\ell_2}$\end{small}. From Theorem \ref{thmtosylvesterrk3}(2), it then follows that for \begin{small}$n= \ell_1\ell_2,$\end{small} or \begin{small}$n= \ell_1^2\ell_2^2$\end{small},  we have \begin{small}$\text{rk } E_{-432n^2}(\Q) \le 3$\end{small}. Further, if we assume that \begin{small}$\dim_{\F_3} \Sh(E_{16n^2}/\Q)[3]$\end{small} is even, then we have \begin{small}$\text{rk } E_{-432n^2}(\Q)=1 \text{ or } 3$\end{small}.

 Using \cite{sg}, we obtain  \begin{small}$\text{rk } E_{-432\ell_1^2\ell_2^2}(\Q)=3$ \end{small} and \begin{small}$\text{rk } E_{-432\ell_1^4\ell_2^4}(\Q)=1$.\end{small}

\begin{theorem}\label{thmtosylvesterrk2}
   Let \begin{small}$n$\end{small} be a cube-free integer co-prime to $3$ having  exactly two distinct prime factors \begin{small}$\ell_1$\end{small} and \begin{small}$\ell_2$.\end{small} In the following cases we have \begin{small}$\text{rk } E_{-432n^2}(\Q) \le 2$.\end{small} 
    \begin{enumerate}
        \item \begin{small}$n \equiv \pm 1 \pmod 9$, $\ell_1 \equiv \ell_2 \equiv 8 \pmod 9$,\end{small}
        \item \begin{small}$n \not\equiv \pm 1 \pmod 9$, $\ell_1 \equiv 1 \pmod 3$, $\ell_2 \equiv 2 \pmod 3$\end{small} and \begin{small}$\ell_2$\end{small} is a cube in \begin{small}$\F_{\ell_1}$,\end{small}
        \item \begin{small}$n \equiv  1 \pmod 9$, $\ell_1 \equiv \ell_2  \equiv 1 \pmod 3$\end{small} and \begin{small}$\ell_1 \equiv  4 \text{ or } 7 \pmod 9$,\end{small}
        \item \begin{small}$n \equiv  1 \pmod 9$, $\ell_1 \equiv \ell_2  \equiv 1 \pmod 9$\end{small} and at least one of \begin{small}$\pi_{\ell_1}, \pi_{\ell_1}'$\end{small} is not a cube in \begin{small}$\F_{\ell_2}$.\end{small}
\end{enumerate}
   Moreover, under the assumption that \begin{small}$\dim_{\F_3} \Sh(E_{16n^2}/\Q)[3]$\end{small} is even in the cases above, we have \begin{small}$\text{rk } E_{-432n^2}(\Q)$\end{small} is either $0$ or $2$.
\end{theorem}

\begin{proof}
Props. \ref{sylvesterpq}-\ref{sylvesterpq4} give \begin{small}$\dim_{\F_3} {\rm Sel}^\phi(E_{16n^2}/K) =3$\end{small} in all the above-mentioned cases.  The statement follows from Lemma \ref{rankbound}. 
\end{proof}

\noindent {\bf Example (1)} Let \begin{small}$\ell_1=53$, $\ell_2=71$\end{small} and \begin{small}$n \in \{\ell_1\ell_2, \ell_1^2\ell_2, \ell_1^2\ell_2^2\}$,\end{small} we have  \begin{small}$n \equiv \pm 1 \pmod 9$\end{small} and \begin{small}$\text{rk } E_{-432n^2}(\Q) \le 2$\end{small} by Theorem \ref{thmtosylvesterrk2}(1). Further, assuming \begin{small}$\dim_{\F_3} \Sh(E_{16n^2}/\Q)[3]$\end{small} is even,  we get \begin{small}$\text{rk } E_{-432n^2}(\Q)=0 \text{ or } 2$\end{small}. 

\cite{sg} gives \begin{small}$\text{rk } E_{-432\ell_1^4\ell_2^2}(\Q)=2$\end{small} and \begin{small}$\text{rk } E_{-432n^2}(\Q)=0$\end{small} for \begin{small}$n \in \{\ell_1\ell_2, \ell_1^2\ell_2^2\}$\end{small}. 

\noindent {\bf Example (2)} Take \begin{small}$\ell_1=37$, $\ell_2=29$\end{small}. Then \begin{small}$n \not\equiv \pm 1 \pmod 9$\end{small} and \begin{small}$\ell_2$\end{small} is a cube in \begin{small}$\F_{\ell_1}$.\end{small} For \begin{small}$n \in \{\ell_1\ell_2, \ell_1^2\ell_2, \ell_1\ell_2^2\}$,\end{small} we have \begin{small}$\text{rk } E_{-432n^2}(\Q) \le 2$\end{small} by Theorem \ref{thmtosylvesterrk2}(2). Further, assuming \begin{small}$\dim_{\F_3} \Sh(E_{16n^2}/\Q)[3]$\end{small} is even,  we get \begin{small}$\text{rk } E_{-432n^2}(\Q)=0 \text{ or } 2$\end{small}. 

Using \cite{sg}, \begin{small}$\text{rk } E_{-432\ell_1^4\ell_2^2}(\Q)=0$\end{small} and \begin{small}$\text{rk } E_{-432n^2}(\Q)=2$\end{small} for \begin{small}$n \in \{\ell_1\ell_2, \ell_1^2\ell_2^2\}$\end{small}. 

\noindent {\bf Example (3)} Take \begin{small}$\ell_1=157$, $\ell_2=193$\end{small} and \begin{small}$n \in \{\ell_1^2\ell_2,  \ell_1\ell_2^2\}$.\end{small} Then we have \begin{small}$\text{rk } E_{-432n^2}(\Q) \le 2$\end{small} by Theorem \ref{thmtosylvesterrk2}(3). Further, the assumption that \begin{small}$\dim_{\F_3} \Sh(E_{16n^2}/\Q)[3]$\end{small} is even implies  \begin{small}$\text{rk } E_{-432n^2}(\Q)=0 \text{ or } 2$\end{small}. 

Using \cite{sg}, \begin{small}$\text{rk } E_{-432\ell_1^4\ell_2^2}(\Q)=2$\end{small} and \begin{small}$\text{rk } E_{-432\ell_1^2\ell_2^4}(\Q)=0$\end{small}.

\noindent {\bf Example (4)} Consider \begin{small}$\ell_1=73$, $\ell_2=19$\end{small}. Since \begin{small}$\ell_1$\end{small} is not a cube modulo \begin{small}$\ell_2$\end{small}, it follows that at least one of \begin{small}$\pi_{\ell_1}, \pi_{\ell_1}'$\end{small} is not a cube in \begin{small}$\F_{\ell_2}$.\end{small} Thus by \ref{thmtosylvesterrk2}(4), for \begin{small}$n= \ell_1\ell_2$\end{small} and \begin{small}$n= \ell_1^2\ell_2^2$\end{small}, we have \begin{small}$\text{rk } E_{-432n^2}(\Q) \le 2$.\end{small} Further, if we assume that \begin{small}$\dim_{\F_3} \Sh(E_{16n^2}/\Q)[3]$\end{small} is even, then we have \begin{small}$\text{rk } E_{-432n^2}(\Q)=0 \text{ or } 2$\end{small}.

From \cite{sg}, we see that \begin{small}$\text{rk } E_{-432\ell_1^2\ell_2^2}(\Q) = 2$ and $\text{rk } E_{-432\ell_1^4\ell_2^4}(\Q)=0$.\end{small}

\begin{theorem}\label{thmtosylvesterrk4}
   Let \begin{small}$n$\end{small} be a cube-free integer co-prime to $3$ having  exactly two distinct prime factors \begin{small}$\ell_1$\end{small} and \begin{small}$\ell_2$.\end{small} If  
     \begin{small}$\ell_1 \equiv \ell_2  \equiv 1 \pmod 9$\end{small} and both \begin{small}$\pi_{\ell_1}, \pi_{\ell_1}'$\end{small} is are cubes in \begin{small}$\F_{\ell_2}$, \end{small} then  \begin{small}$\text{rk } E_{-432n^2}(\Q) \le 4$.\end{small}
     
     Further, assume that \begin{small}$\dim_{\F_3} \Sh(E_{16n^2}/\Q)[3]$\end{small} is even. Then we have \begin{small}$\text{rk } E_{-432n^2}(\Q) \in \{0,2,4\}$.\end{small}   
\end{theorem}

\begin{proof}
From Prop. \ref{sylvesterpq4}, we see that that \begin{small}$\dim_{\F_3} {\rm Sel}^\phi(E_{16n^2}/K) =5$\end{small} in the above mentioned case. The statement follows from Lemma \ref{rankbound}. 
\end{proof}

\noindent {\bf Example} Let \begin{small}$\ell_1=199$, $\ell_2=109$\end{small}. Notice that \begin{small}$\ell_1$\end{small} is a cube in \begin{small}$\F_{\ell_2}$\end{small}. Next observe that there exists an \begin{small}$\alpha \in \Z_{\ell_2}$\end{small} such that \begin{small}$\alpha \equiv 18 \pmod{109}$\end{small}. Writing \begin{small}$\pi_{\ell_1} = \frac{1}{2}(11+ 15 \alpha)$\end{small}, we see that \begin{small}$\pi_{\ell_1} \equiv -23 \pmod{109} \equiv (-11)^3 \pmod{109}$\end{small}. As a consequence, both \begin{small}$\pi_{\ell_1}, \pi_{\ell_1}'$\end{small} are cubes in \begin{small}$\F_{\ell_2}$\end{small}. From Theorem \ref{thmtosylvesterrk4}, we have \begin{small}$\text{rk } E_{-432n^2}(\Q) \le 4$\end{small} for \begin{small}$n \in \{\ell_1\ell_2,  \ell_1^2\ell_2, \ell_1\ell_2^2, \ell_1^2\ell_2^2\}$\end{small}.  Further, if we assume that \begin{small}$\dim_{\F_3} \Sh(E_{16n^2}/\Q)[3]$\end{small} is even, then we have \begin{small}$\text{rk } E_{-432n^2}(\Q)=0 \text{ or } 2 \text{ or } 4$\end{small}.

From \cite{sg}, we see that \begin{small}$\text{rk } E_{-432\ell_1^2\ell_2^2}(\Q)=4$, $\text{rk } E_{-432\ell_1^4\ell_2^2}(\Q)=2$ and $\text{rk } E_{-432\ell_1^2\ell_2^4}(\Q) = \text{rk } E_{-432\ell_1^4\ell_2^4}(\Q)=0$.\end{small}

\begin{rem}
For a prime \begin{small}$\ell \equiv 1 \pmod 9$,\end{small} Sylvester's conjecture is unable to pinpoint which integers \begin{small}$n$\end{small} of the form \begin{small}$\ell$\end{small} or \begin{small}$\ell^2$\end{small} are rational cube sum. This question was studied by \cite{rz}. The Birch and Swinnerton-Dyer conjecture predicts that 
\begin{small}$$ L(E_{16n^2}/\Q, 1) = * C_n,$$\end{small}
where $*$ is an non-zero number (depending on \begin{small}$n$\end{small}) and \begin{small}$C_n = 0$\end{small} if and only if \begin{small}$n$\end{small} is a rational cube sum (note that if \begin{small}$n$\end{small} is not a rational cube sum, then one expects that \begin{small}$C_n \neq 0$,\end{small} and  \begin{small}$C_n = |\Sh(E_{16n^2}/\Q)|$\end{small} which is expected to be a square integer). \cite{rz} provided efficient methods to numerically test for which of these integers \begin{small}$n$\end{small} of the form \begin{small}$\ell$\end{small} or \begin{small}$\ell^2$\end{small} with \begin{small}$\ell \equiv 1 \pmod 9$, $C_n=0.$\end{small}

In a similar fashion, one can ask: 
Let \begin{small}$n$\end{small} be a cube-free integer co-prime to $3$ with exactly two distinct prime factors \begin{small}$\ell_1$\end{small} and \begin{small}$\ell_2$.\end{small} Can you provide an efficient method to numerically test which of the following integers \begin{small}$n$\end{small} are rational cube sum?
\begin{itemize}
    \item \begin{small}$n \equiv 1 \pmod 9$\end{small} with \begin{small}$\ell_1 \equiv \ell_2 \equiv 1 \pmod 3$,\end{small}
    \item \begin{small}$n \equiv \pm 1 \pmod 9$\end{small} with \begin{small}$\ell_1 \equiv  \ell_2 \equiv 8 \pmod 9$,\end{small}
    \item \begin{small}$n \not\equiv \pm 1 \pmod 9$\end{small} with \begin{small}$\ell_1 \equiv -\ell_2 \equiv 1 \pmod 3$\end{small} and \begin{small}$\ell_2$\end{small} a cube in \begin{small}$\F_{\ell_1}$.\end{small}
\end{itemize}
\end{rem}

\end{document}